\documentclass[11pt, reqno]{amsart}
\usepackage{amssymb, amsthm, amsmath, amsfonts,mathtools}
\usepackage{array, epsfig}
\usepackage{bbm}
\usepackage{enumitem}

\usepackage{hyperref}
\usepackage[numbers,square]{natbib}
\usepackage{color}
\allowdisplaybreaks

\numberwithin{equation}{section}

  \evensidemargin
\oddsidemargin
\newcommand{\stirling}[2]{\genfrac{[}{]}{0pt}{}{#1}{#2}}
\newcommand{\stirlingsec}[2]{\genfrac{\{}{\}}{0pt}{}{#1}{#2}}

\newcommand{\stirlingb}[2]{B\hspace*{-0.2mm}\big[{#1},{#2}\big]}
\newcommand{\stirlingsecb}[2]{B\hspace*{-0.2mm}\big\{{#1},{#2}\big\}}

\newcommand{\N}{\mathbb{N}}

\newcommand{\R}{\mathbb{R}}
\newcommand{\C}{\mathbb{C}}

\newcommand{\F}{\mathcal{F}}

\newcommand{\PP}{\mathbb{P}}

\newcommand{\eps}{\varepsilon}

\newcommand*\xbar[1]{%
   \hbox{%
     \vbox{%
       \hrule height 0.5pt 
       \kern0.25ex
       \hbox{%
         \kern-0.05em
         \ensuremath{#1}%
         \kern-0.1em
       }%
     }%
   }%
}

\DeclareMathOperator{\E}{\mathbb{E}}

\DeclareMathOperator{\lin}{lin}
\DeclareMathOperator{\conv}{conv}
\DeclareMathOperator{\pos}{pos}

\DeclareMathOperator{\relint}{relint}

\newcommand{\bP}{\mathbb{P}}

\newcommand{\cF}{\mathcal{F}}

\renewcommand{\P}{\mathbb{P}}

\newcommand{\aff}{\mathop{\mathrm{aff}}\nolimits}

\newcommand{\eqdistr}{\stackrel{d}{=}}

\theoremstyle{plain}
\newtheorem{theorem}{Theorem}[section]
\newtheorem{lemma}[theorem]{Lemma}
\newtheorem{corollary}[theorem]{Corollary}
\newtheorem{proposition}[theorem]{Proposition}

\theoremstyle{definition}

\theoremstyle{remark}

\makeindex
\makeglossary

\begin{document}

\author{Thomas Godland}
\address{Thomas Godland: Institut f\"ur Mathematische Stochastik,
Westf\"alische Wilhelms-Universit\"at M\"unster,
Orl\'eans-Ring 10,
48149 M\"unster, Germany}
\email{t\_godl01@uni-muenster.de}

\author{Zakhar Kabluchko}
\address{Zakhar Kabluchko: Institut f\"ur Mathematische Stochastik,
Westf\"alische {Wilhelms-Uni\-ver\-sit\"at} M\"unster,
Orl\'eans--Ring 10,
48149 M\"unster, Germany}
\email{zakhar.kabluchko@uni-muenster.de}


\title[Angle sums of Schl\"afli orthoschemes]{Angle sums of Schl\"afli orthoschemes}

\keywords{Schl\"afli orthoschemes, Weyl chambers, polytopes, polyhedral cones, solid angles, conic intrinsic volumes, Stirling numbers, random walks}

\subjclass[2010]{Primary: 60D05, 11B73.  Secondary: 51F15, 52A55, 52A22}


\begin{abstract}
We consider the simplices 
\begin{align*}
K_n^A=\{x\in\R^{n+1}:x_1\ge x_2\ge \ldots\ge x_{n+1},x_1-x_{n+1}\le 1,x_1+\ldots+x_{n+1}=0\}
\end{align*}
and
\begin{align*}
K_n^B=\{x\in\R^n:1\ge x_1\ge x_2\ge \ldots\ge x_n\ge 0\},
\end{align*}
which are called the \textit{Schl\"afli orthoschemes of types $A$ and $B$}, respectively.  
We describe the tangent cones at their $j$-faces and compute explicitly the sums of the conic intrinsic volumes 
of these tangent cones at all $j$-faces of $K_n^A$ and $K_n^B$. This setting contains sums of external and internal angles of $K_n^A$ and $K_n^B$ as special cases. The sums are evaluated in terms of Stirling numbers of both kinds.  We generalize these results to finite products of Schl\"afli orthoschemes of type $A$ and $B$ and, as a probabilistic consequence,  derive formulas for the expected number of $j$-faces of the Minkowski sums of the convex hulls of a  finite number of Gaussian random walks and random bridges.
Furthermore, we evaluate the analogous angle sums for the tangent cones of Weyl chambers of types $A$ and $B$ and finite products thereof.
\end{abstract}

\maketitle
\section{Introduction}

The Schl\"afli orthoscheme of type $B$ in $\R^n$, denoted by $K_n^B$, is the simplex spanned by the $n+1$ vertices
\begin{align*}
(0,0,0,\ldots,0),(1,0,0,\ldots,0),(1,1,0,\ldots,0),\ldots,(1,1,1,\ldots,1)
\end{align*}
or, equivalently,
\begin{align*}
K_n^B=\{x\in\R^n:1\ge x_1\ge x_2\ge \ldots\ge x_n\ge 0\}.
\end{align*}
The classical intrinsic volumes of $K_n^B$ were computed by Gao and Vitale~\cite{GV01} in order to evaluate the intrinsic volumes of the so-called Brownian motion body.

The Schl\"afli orthoscheme of type $A$ in $\R^{n+1}$, denoted by $K_n^A$, was studied by Gao~\cite{Gao2003} in the context of Brownian bridges and is defined as the simplex spanned by the vertices $P_0,\ldots,P_{n+1}$, where $P_0:=(0,\ldots,0)$ and
\begin{align*}
P_i=(\underbrace{1,\ldots,1}_{i},0,\ldots,0)-\frac{i}{n+1}(1,1,\ldots,1)
\end{align*}
for $i=1,\ldots,n+1$. Equivalently, the Schl\"afli orthoscheme $K_n^A$ can be expressed as
\begin{align*}
K_n^A=\{x\in\R^{n+1}:x_1\ge x_2\ge \ldots\ge x_{n+1},x_1-x_{n+1}\le 1,x_1+\ldots+x_{n+1}=0\}.
\end{align*}
We will present further details on the probabilistic meaning of  Schl\"afli orthoschemes in Section~\ref{sec:schlaefli_orthoschemes}.

In the present paper, we evaluate certain angle sums or, more generally, sums of conic intrinsic volumes, of the Schl\"afli orthoschemes. For a polytope $P$, let $\cF_j(P)$ denote the set of all $j$-dimensional faces of $P$. The tangent cone of $P$ at its $j$-dimensional face $F$ is the convex cone $T_F(P)$ defined by
\begin{align*}
T_F(P)=\{x\in\R^n: v+\eps x\in P\text{ for some }\eps>0\},
\end{align*}
where $v$ is any point in $F$ not belonging to a face of smaller dimension. We explicitly compute the conic intrinsic volumes of the tangent cones of the Schl\"afli orthoschemes at their $j$-dimensional faces and, in particular, the sum of the conic intrinsic volumes over all such faces. The $k$-th conic intrinsic volume of a convex cone $C$, denoted by $\upsilon_k(C)$,   is a spherical or conic  analogue of the usual intrinsic volume of a convex set and will be formally introduced in Section~\ref{section:integral_geometry}.  Among other results, we will show that
\begin{align}\label{eq:introd_anglesums}
\sum_{F\in\cF_j(K_n^A)}\upsilon_k(T_{F}(K_n^A))
=
\sum_{F\in\cF_j(K^B_n)}\upsilon_k(T_{F}(K_n^B))
=\frac{j!}{n!}\stirling{n+1}{k+1}\stirlingsec{k+1}{j+1},
\end{align}
where the numbers $\stirling{n}{k}$ and $\stirlingsec{n}{k}$ are the Stirling numbers of the first and second kind, respectively.

Furthermore, we will compute the analogous angle (and conic intrinsic volume) sums for the tangent cones of \textit{Weyl chambers} of type $A$ and $B$ which are convex cones in  $\R^n$ defined by
\begin{align*}
A^{(n)}:=\{x\in\R^n:x_1\ge x_2\ge \ldots\ge x_n\}
\;\;
\text{ and }
\;\;
B^{(n)}:=\{x\in\R^n:x_1\ge x_2\ge \ldots\ge x_n\ge 0\}.
\end{align*}
The corresponding formulas are given by
\begin{align}\label{eq:introd_anglesums_Weyl chambers}
\sum_{F\in\cF_j(A^{(n)})}\upsilon_k(T_{F}(A^{(n)}))=
\frac{j!}{n!}\stirling{n}{k}\stirlingsec{k}{j},
\quad
\sum_{F\in\cF_j(B^{(n)})}\upsilon_k(T_{F}(B^{(n)}))
=\frac{2^jj!}{2^nn!}\stirlingb nk \stirlingsecb kj,
\end{align}
where the numbers $\stirlingb nk$ and $\stirlingsecb nk$ denote the $B$-analogues of the Stirling numbers of the first and second kind, respectively, which we will formally introduce in Section~\ref{section:stirling_numbers}. An application of Equations~\eqref{eq:introd_anglesums} and~\eqref{eq:introd_anglesums_Weyl chambers} to a problem of compressed sensing will be given in Section~\ref{section:compressed_sensing}.
Observe that in the special cases $k=n$ and $k=j$, Equations~\eqref{eq:introd_anglesums} and~\eqref{eq:introd_anglesums_Weyl chambers} yield formulas for the sums of the internal and external angles of Schl\"afli orthoschemes and Weyl chambers. 

We will generalize the above results to finite products of Schl\"afli orthoschemes and finite products of Weyl chambers leading to rather complicated formulas in terms of coefficients in the Taylor expansion of a certain function. The main results on angle and conic intrinsic volume sums will be stated in Section~\ref{section:angle_sums}.
%
As a probabilistic interpretation of these results,  we consider convex hulls of Gaussian random walks and random bridges in Section~\ref{section:facenumbers_walks_bridges}. The expected numbers of $j$-faces of the convex hull of a single Gaussian random walk or a Gaussian bridge in $\R^d$ (even in a more general non-Gaussian setting) were already evaluated in~\cite{KVZ17}.
Our general result on the angle sums of products of Schl\"afli orthoschemes yields a formula for the expected number of $j$-faces of the Minkowski sum of several convex hulls of Gaussian random walks or Gaussian random bridges.

It turns out that the tangent cones of the Schl\"afli orthoschemes (and of the Weyl chambers) are essentially products of Weyl chambers of type $A$ and $B$. We will derive~\eqref{eq:introd_anglesums} and~\eqref{eq:introd_anglesums_Weyl chambers} as a special case of a more general Proposition~\ref{prop:sum_intr_vol} stated in Section~\ref{section:method_of_proof}. This proposition gives a formula for the sum of the conic intrinsic volumes of a product of Weyl chambers in terms of the generalized Stirling numbers of the first and second kind. The main ingredients in the proof of this proposition are the known formulas for the conic intrinsic volumes of Weyl chambers; see e.g.~\cite[Theorem 4.2]{KVZ15} or \cite[Theorem 1.1]{GK20_Intr_Vol}.
We will also present an alternative proof of Proposition~\ref{prop:sum_intr_vol} in Section~\ref{section:proofs_intext_angles}, where the main ingredient is an explicit evaluation of the internal and external angles of the tangent cones and their faces.


\section{Preliminaries}\label{section:prelimiaries}

In this section we collect notation and facts from convex geometry and combinatorics. The reader may skip this section at first reading and return to it when necessary.

\subsection{Facts from convex geometry}\label{section:convex_geometry}
For a set $M\subset\R^n$ denote by $\lin M$ (respectively, $\aff M$) its linear (respectively, affine) hull, that is, the minimal linear (respectively, affine) subspace containing $M$. Equivalently, $\lin M$ (respectively, $\aff M$) is the set of all linear (respectively, affine) combinations of elements of $M$. The relative interior of $M$, denoted by $\relint M$, is the set of interior points of $M$ relative to its affine hull $\aff M$.
Let also $\conv M$ denote the convex hull of $M$ which is defined as the minimal convex set containing $M$, or equivalently
\begin{align*}
\conv M:=\Big\{\sum_{i=1}^m\lambda_ix_i:m\in\N,x_1,\ldots,x_m\in M,\lambda_1,\ldots,\lambda_m\ge 0, \lambda_1+\ldots+\lambda_m=1\Big\}.
\end{align*}
Similarly, let $\pos M$ denote the positive (or conic) hull of $M$:
\begin{align*}
\pos M:=\Big\{\sum_{i=1}^m\lambda_ix_i:m\in\N,x_1,\ldots,x_m\in M,\lambda_1,\ldots,\lambda_m\ge 0\Big\}.
\end{align*}

A set $C\subset \R^n$ is called a \textit{(convex) cone} if $\lambda_1x_1+\lambda_2x_2\in C$ for all $x_1,x_2\in C$ and $\lambda_1,\lambda_2\ge 0$. Thus, $\pos M$ is the minimal cone containing $M$. The \textit{dual cone} of a cone $C\subset\R^n$ is defined as
\begin{align*}
C^\circ=\{x\in\R^n:\langle x,y\rangle\le 0\; \text{ for all } y\in C\},
\end{align*}
where $\langle\cdot,\cdot\rangle$ denotes the Euclidean scalar product.  We will make use of the following simple duality relation that holds for arbitrary $x_1,\ldots,x_m\in\R^n$:
\begin{align}\label{eq:duality_relation}
\pos\{x_1,\ldots,x_m\}^\circ=\{v\in\R^n:\langle v,x_i\rangle\le 0 \text{ for all } i=1,\ldots,m\}.
\end{align}

A \textit{polyhedral set} $P\subset\R^d$ is an intersection of finitely many closed half-spaces (whose boundaries need not pass through the origin). A bounded polyhedral set is called \textit{polytope}. A \textit{polyhedral cone} is an intersection of finitely many closed half-spaces whose boundaries contain the origin and therefore a special case of a polyhedral set. The faces of $P$ (of arbitrary dimension) are obtained by replacing some of the half-spaces, whose intersection defines the polyhedral set, by their boundaries and taking the intersection. For $k\in\{0,1,\ldots,d\}$, we denote the set of $k$-dimensional faces of a polyhedral set $P$ by $\cF_k(P)$. Furthermore, we denote the number of $k$-faces of $P$ by $f_k(P):=|\cF_k(P)|$. The \textit{tangent cone} of $P$ at a face $F\in\cF_k(P)$ is defined by
\begin{align*}
T_F(P)=\{x\in\R^n:v+\eps x\in P\text{ for some }\eps>0\},
\end{align*}
where $v$ is any point in the relative interior of $F$. It is known that this definition does not depend on the choice of $v$. The \textit{normal cone} of $P$ at the face $F$ is defined as the dual of the tangent cone, that is
\begin{align*}
N_F(P)=T_F(P)^\circ=\{w\in\R^n:\langle w,u\rangle\le 0 \text{ for all } u\in T_F(P)\}.
\end{align*}
It is easy to check that given a face $F$ of a cone $C$, the corresponding normal cone $N_F(C)$ satisfies $N_F(C):=(\lin F)^\perp\cap C^\circ$, where $L^\perp$ denotes the orthogonal complement of a linear subspace $L$.

\subsection{Conic intrinsic volumes and angles of polyhedral sets}\label{section:integral_geometry}

Now let us introduce some geometric functionals of cones that we are going to consider. The following facts are mostly taken from~\cite[Section 2]{Amelunxen2017}; see also~\cite[Section~6.5]{SW08}.  At first, we define the conical intrinsic volumes which are the analogues of the usual intrinsic volumes in the setting of conic or spherical geometry.

Let $C\subset \R^n$ be a polyhedral cone, and $g$ be an $n$-dimensional standard Gaussian random vector. Then, for $k\in \{0,\ldots,n\}$, the $k$-th \textit{conic intrinsic volume} of $C$ is defined by
\begin{align*}
\upsilon_k(C):=\sum_{F\in\F_k(C)}\PP(\Pi_C(g)\in\relint F).
\end{align*}
Here, $\Pi_C$ denotes the metric projection on $C$, that is $\Pi_C(x)$ is the vector in $C$ minimizing the Euclidean distance to $x\in\R^n$.
The conic intrinsic volumes of a cone $C$ form a probability distribution on $\{0,1,\ldots,\dim C\}$, that is
\begin{align}\label{eq:intrinsic_sum1} 
\sum_{k=0}^{\dim C} \upsilon_k(C)=1\quad\text{and}\quad \upsilon_k(C)\ge 0,\quad k=0,\ldots,\dim C.
\end{align}
The Gauss-Bonnet formula~\citep[Corollary 4.4]{Amelunxen2017}  states that
\begin{align}\label{eq:intrinsic_sum2}
\sum_{k=0}^{\dim C} (-1)^k\upsilon_k(C)=0
\end{align}
for every cone $C$ which is not a linear subspace, which implies that
\begin{align}\label{eq:sum_odd_Intr_vol}
\sum_{k=1,3,5,\ldots}\upsilon_k(C)=\sum_{k=0,2,4,\ldots}\upsilon_k(C)=\frac{1}{2}.
\end{align}
Furthermore, the conic intrinsic volumes satisfy the product rule
\begin{align*}
\upsilon_k(C_1\times \ldots\times C_m)=\sum_{i_1+\ldots+i_m=k}\upsilon_{i_1}(C_1)\cdot\ldots\cdot\upsilon_{i_m}(C_m),
\end{align*}
where $C_1\times\ldots\times C_m$ is the Cartesian product of $C_1,\ldots,C_m$.
The product rule implies that the generating polynomial of the intrinsic volumes of $C$, defined by $P_C(t):=\sum_{k=0}^{\dim C}\upsilon_k(C)t^k$, satisfies
\begin{align}\label{eq:intr_vol_prod_gen_fct}
P_{C_1\times\ldots\times C_m}(t)=P_{C_1}(t)\cdot\ldots\cdot P_{C_m}(t).
\end{align}
For example,  for an $i$-dimensional linear subspace $L$, 
we have $\upsilon_k(C\times L)=\upsilon_{k-i}(C)$ for $k\geq i$. 

The \textit{solid angle} (or just angle) of a cone $C\subset \R^n$ is defined as
\begin{align*}
\alpha(C):=\P(Z\in C),
\end{align*}
where $Z$ is uniformly distributed on the unit sphere in the linear hull $\lin C$. Equivalently, we can take a random vector $Z$ having a standard Gaussian distribution on the ambient linear subspace $\lin C$.
For a $d$-dimensional cone $C\subset\R^n$, where $d\in\{1,\ldots,n\}$, the $d$-th conical intrinsic volume coincides with the solid angle of $C$, that is
$$
\upsilon_d(C) = \alpha (C).
$$

The \textit{internal angle} of a polyhedral set $P$ at a face $F$ is defined as the solid angle of its tangent cone:
\begin{align*}
\beta(F,P):=\alpha(T_F(P)).
\end{align*}
The \textit{external angle} of $P$ at a face $F$ is defined as the solid angle of the normal cone of $F$ with respect to $P$, that is
\begin{align*}
\gamma(F,P):=\alpha(N_F(P)).
\end{align*}

The conic intrinsic volumes of a cone $C\subset \R^n$ can be computed in terms of the internal and external angles of its faces as follows:
\begin{align}\label{eq:in_vol_internalexternal}
\upsilon_k(C)
=
\sum_{F\in\F_k(C)}\alpha(F)\alpha(N_F(C)),
\qquad
k\in \{0,\ldots,n\}.
\end{align}

Let $W_{n-k}\subset\R^n$ be random linear subspace having the uniform distribution on the Grassmann manifold of all $(n-k)$-dimensional linear subspaces. Then, following Grünbaum~\cite{gruenbaum_grass_angles}, the \textit{Grassmann angle} $\gamma_k(C)$ of a cone $C\subset\R^n$ is defined as
\begin{align}\label{eq:grassmann_angle_def}
\gamma_k(C):=\PP(W_{n-k}\cap C\neq\{0\}),
\end{align}
for $k\in \{0,\ldots,n\}$.
The Grassmann angles do not depend on the dimension of the ambient space, that is, if we embed $C$ in $\R^N$ where $N\ge n$, the Grassmann angle will be the same. If $C$ is not a linear subspace, then $\frac{1}{2}\gamma_k(C)$ is also known as the \textit{$k$-th conic quermassintegral} $U_k(C)$ of $C$, see~\cite[(1)-(4)]{HugSchneider2016}, or as the \textit{half-tail functional} $h_{k+1}(C)$, see~\cite{amelunxen_edge}. The conic intrinsic volumes and the Grassmann angles are known to satisfy the linear relation
\begin{align}\label{eq:relation_grass_intr}
\gamma_k(C)=2\sum_{i=1,3,5,\ldots}\upsilon_{k+i}(C),
\end{align}
see~\cite[(2.10)]{Amelunxen2017}, provided $C$ is not a linear subspace.

\subsection{Stirling numbers and their generating functions}\label{section:stirling_numbers}

In this section, we are going to recall the definitions of  various kinds of Stirling numbers and their generating functions. As mentioned in the introduction, these numbers appear in various results presented in this paper. 

The (signless) \textit{Stirling number of the first kind} $\stirling{n}{k}$ is defined as the number of permutations of the set $\{1,\ldots,n\}$ having exactly $k$ cycles. Equivalently, these numbers can be defined as the coefficients of the polynomial
\begin{align}\label{eq:def_stirling1_polynomial}
t(t+1)\cdot\ldots \cdot (t+n-1)=\sum_{k=0}^n\stirling{n}{k}t^k
\end{align}
for $n\in\N_0$, with the convention that $\stirling{n}{k}=0$ for $n\in \N_0$, $k\notin\{0,\ldots,n\}$, and $\stirling{0}{0} = 1$. By~\cite[Equations~(1.9),(1.15)]{Pitman2006}, the Stirling numbers of the first kind can also be represented as the following sum:
\begin{align}\label{eq:rep_stirling1_sum}
\stirling{n}{k}=\frac{n!}{k!}\sum_{\substack{m_1,\ldots,m_k\in\N:\\m_1+\ldots+m_k=n}}\frac{1}{m_1\cdot\ldots\cdot m_k}.
\end{align}

The \textit{$B$-analogues of the Stirling numbers of the first kind}, denoted by $\stirlingb nk$, are defined as the coefficients of the polynomial
\begin{align}\label{eq:def_stirling1b}
(t+1)(t+3)\cdot\ldots\cdot (t+2n-1)=\sum_{k=0}^n \stirlingb nk t^k
\end{align}
for $n\in\N_0$ and, by convention, $\stirlingb nk=0$ for $k\notin\{0,\ldots,n\}$. These numbers appear as Entry A028338 in~\cite{sloane}. The exponential generating functions of the arrays $(\stirling{n}{k})_{n,k\ge 0}$ and $(\stirlingb nk)_{n,k\ge 0}$ are given by
\begin{align}\label{eq:gen_fct_stirling1}
\sum_{n=k}^\infty \stirling{n}{k}\frac{t^n}{n!}=\frac{1}{k!}\big(-\log(1-t)\big)^k,\quad \sum_{k=0}^\infty\sum_{n=k}^\infty\stirling{n}{k}\frac{t^n}{n!}y^k=(1-t)^{-y}
\end{align}
and
\begin{align}\label{eq:gen_fct_stirling1b}
\sum_{k=0}^\infty\sum_{n=k}^\infty \stirlingb nk\frac{t^n}{n!}y^k=(1-2t)^{-\frac{1}{2}(y+1)};
\end{align}
see~\cite[Proposition 2.3]{GK20_Intr_Vol} for the proof of~\eqref{eq:gen_fct_stirling1b}.

The \textit{Stirling number of the second kind} $\stirlingsec{n}{k}$  is defined as the number of partitions of the set $\{1,\ldots,n\}$ into $k$ non-empty subsets. Similarly to~\eqref{eq:rep_stirling1_sum}, the Stirling numbers of the second kind can be represented as the following sum:
\begin{align}\label{eq:rep_stirling2_sum}
\stirlingsec{n}{k}=\frac{n!}{k!}\sum_{\substack{m_1,\ldots,m_k\in\N:\\m_1+\ldots+m_k=n}}\frac{1}{m_1!\cdot\ldots\cdot m_k!};
\end{align}
see~\cite[Equations~(1.9),(1.13)]{Pitman2006}.
The \textit{$B$-analogues of the Stirling numbers of the second kind}, denoted by $\stirlingsecb nk$, are defined as
\begin{align*}
\stirlingsecb nk=\sum_{m=k}^{n}2^{m-k}\binom{n}{m}\stirlingsec{m}{k}.
\end{align*}
They appear as Entry A039755 in~\cite{sloane} and were studied by Suter~\cite{suter}.

The exponential generating functions of the arrays $(\stirlingsec{n}{k})_{n,k\ge 0}$ and $(\stirlingsecb nk)_{n,k\ge 0}$ are given by
\begin{align}\label{eq:gen_fct_stirling2}
\sum_{n=k}^\infty\stirlingsec{n}{k}\frac{t^n}{n!}=\frac{1}{k!}(e^t-1)^k,\quad\sum_{k=0}^\infty\sum_{n=k}^\infty\stirlingsec{n}{k}\frac{t^n}{n!}y^k=e^{(e^t-1)y}
\end{align}
and
\begin{align}\label{eq:gen_fct_stirling2b}
\sum_{k=0}^\infty \sum_{n=k}^\infty \stirlingsecb nk\frac{t^n}{n!}y^k=e^{\frac{y}{2}(e^{2t}-1)}e^{t};
\end{align}
see~\cite[Theorem 4]{suter} for~\eqref{eq:gen_fct_stirling2b}. The numbers $\stirlingsecb nk$ and $\stirlingsec{n}{k}$ appear as coefficients in the formulas
\begin{align*}
t^n=\sum_{k=0}^n(-1)^{n-k}\stirlingsecb nk(t+1)(t+3)\ldots(t+2k-1),\;\; t^n=\sum_{k=0}^n(-1)^{n-k}\stirlingsec{n}{k}t(t+1)\ldots(t+k-1);
\end{align*}
see Entry A039755 in~\cite{sloane} and also~\cite{bagno_biagioli_garber_some_identities,bagno_garber_balls} for combinatorial proofs of both identities, which should be compared to the formulas~\eqref{eq:def_stirling1_polynomial} and~\eqref{eq:def_stirling1b} for $\stirling{n}{k}$ and their $B$-analogues $\stirlingb nk$.

More generally, it is possible to  define the $r$-Stirling numbers of the first and the second kind. For $r\in\N_0$, the (signless) \textit{r-Stirling number of the first kind}, denoted by $\stirling{n}{k}_r$, is defined as the number of permutations of the set $\{1,\ldots,n\}$ having $k$ cycles such that the numbers $1,2,\ldots,r$ are in distinct cycles; see \cite[(1)]{Broder_1982}. The \textit{r-Stirling number of the second kind}, denoted by $\stirlingsec{n}{k}_r$, is defined as the number of partitions of the set $\{1,\ldots,n\}$ into $k$ non-empty disjoint subsets such that the numbers $1,2,\ldots,r$ are in distinct subsets; see \cite[(2)]{Broder_1982}. Obviously, for $r\in\{0,1\}$, the $r$-Stirling numbers of the first and second kind coincide with the classical Stirling numbers, respectively. The $r$-Stirling numbers were introduced by Carlitz~\cite{Carlitz1980,Carlitz2_1980} under the name weighted Stirling numbers.

The exponential generating functions in one and two variables of the $r$-Stirling numbers of the first kind are given by
\begin{align}\label{eq:gen_fct_r-stirling1}
\sum_{n=k}^\infty\stirling{n+r}{k+r}_r \frac{t^n}{n!}=\frac{1}{k!}(1-t)^{-r}\big(-\log(1-t)\big)^k,\quad
\sum_{k=0}^\infty\sum_{n=k}^\infty\stirling{n+r}{k+r}_r \frac{t^n}{n!}y^k=\Big(\frac{1}{1-t}\Big)^{r+y};
\end{align}
see \cite[Equations~(36),(37)]{Broder_1982}. For the $r$-Stirling numbers of the second kind they are given by
\begin{align}\label{eq:gen_fct_r-stirling2}
\sum_{n=k}^\infty\stirlingsec{n+r}{k+r}_r \frac{t^n}{n!}=\frac{1}{k!}e^{rt}(e^{t}-1)^k,\quad
\sum_{k=0}^\infty\sum_{n=k}^\infty\stirlingsec{n+r}{k+r}_r \frac{t^n}{n!}y^k=e^{y(e^t-1)}e^{rt};
\end{align}
see \cite[Equations~(38),(39)]{Broder_1982}.
The $r$-Stirling numbers can equivalently be defined in terms of the regular Stirling numbers by
\begin{align}\label{def:r-Stirling1}
\stirling{n}{k}_r=\sum_{m=0}^{n-k}\binom{n-r}{m}\stirling{n-r-m}{k-r}r^{\overline{m}},
\end{align}
where $r^{\overline{m}}:=r(r+1)\cdot\ldots\cdot (r+m-1)$ denotes the rising factorial, $r^{\overline{0}}:=1$, and
\begin{align}\label{def:r-Stirling2}
\stirlingsec{n}{k}_r=\sum_{m=k-r}^{n-r}\binom{n-r}{m}\stirlingsec{m}{k-r}r^{n-r-m};
\end{align}
see~\cite[Equations~(27),(32)]{Broder_1982}. This even yields an analytic continuation of the $r$-Stirling numbers to non-integer (arbitrary complex) $r$, given by
\begin{align}\label{eq:def_analy_r-stirling1}
\stirling{n+r}{k+r}_r=\sum_{m=0}^{n-k}\binom{n}{m}\stirling{n-m}{k}r^{\overline{m}}
\end{align}
and
\begin{align}\label{eq:def_analy_r-stirling2}
\stirlingsec{n+r}{k+r}_r=\sum_{m=k}^n\binom{n}{m}\stirlingsec{m}{k}r^{n-m}.
\end{align}
Note the following special values:
\begin{align*}
\stirling{n}{n}_r=\stirlingsec{n}{n}_r=1,\quad n\ge r\quad\text{and}\quad
\stirling{n}{k}_r=\stirlingsec{n}{k}_r=0,\quad k\notin\{r,\ldots,n\}.
\end{align*}
For $r=1/2$, we observe the following relations between the $r$-Stirling numbers and the numbers $\stirlingb nk$ and $\stirlingsecb nk$:
\begin{align}\label{eq:rel_1/2-Stirling_B-analogues}
\stirling{n+1/2}{k+1/2}_{1/2}=2^{k-n}\stirlingb nk,\quad \stirlingsec{n+1/2}{k+1/2}_{1/2}=2^{k-n}\stirlingsecb nk.
\end{align}
Both can  easily be verified  by comparing the generating functions.
Let us also mention that besides the $r$-Stirling numbers there is another construction, the generalized two-parameter Stirling numbers~\cite{lang_stirling} (see also~\cite{bala_stirling}) containing Stirling numbers of  both types and their $B$-analogues as special cases corresponding to the parameters $(d,a) = (1,0)$ and $(d,a) = (2,1)$, respectively. We will not use this general construction here.

\section{Main results}\label{section:main_results}
\subsection{Schl\"afli orthoschemes}\label{sec:schlaefli_orthoschemes}
The  polytopes we are interested in this paper are called \textit{Schl\"afli orthoschemes}. As mentioned in the introduction, the Schl\"afli orthoscheme of type $B$ in $\R^n$ is defined as
\begin{align*}
K_{n}^B:
=&	\conv\{(0,0,\ldots,0),(1,0,\ldots,0),(1,1,0,\ldots,0),\ldots,(1,1,\ldots,1)\}\\
=&	\{x\in\R^n:1\ge x_1\ge x_2\ge\ldots\ge x_n\ge 0\}.
\end{align*}
Note that for convenience, we set $K_0^B:=\{0\}$.

Similarly, the Schl\"afli orthoscheme of type $A$ in $\R^{n+1}$ is defined as the convex hull of the $(n+1)$-dimensional vectors $P_0,P_1,\ldots,P_{n+1}$, where $P_0=(0,0,\ldots,0)$ and
\begin{align*}
P_i=(\underbrace{1,\ldots,1}_{i},0,\ldots,0)-\frac{i}{n+1}(1,1,\ldots,1),\quad 1\le i\le n+1.
\end{align*}
It is not difficult to check  that
\begin{align*}
K_n^A=\{x\in\R^{n+1}:x_1\ge x_2\ge\ldots\ge x_{n+1},x_1-x_{n+1}\le 1, x_1+\ldots+x_{n+1}=0\}.
\end{align*}
Again, we put $K_0^A=\{0\}$. The index shift from type $B$ to type $A$ will turn out to be convenient since $K_n^A\subset\R^{n+1}$ is an $n$-dimensional polytope. In fact, the Schl\"afli orthoschemes of type $A$ and type $B$ are simplices since they are convex hulls of $n+1$ affinely independent vectors. The Schl\"afli orthoscheme of type $B$ was already considered by Gao and Vitale~\cite{GV01} who among other things evaluated the classical intrinsic volumes of $K_n^B$. Similar calculations for the Schl\"afli orthoscheme of type $A$ were made by Gao~\cite{Gao2003}.

The definition of the Schl\"afli orthoscheme can be motivated by a connection to random walks and random bridges. In fact, consider Gaussian random matrices $G_B\in\R^{d\times n}$ and $G_A\in\R^{d\times(n+1)}$, that is, random matrices having independent and standard Gaussian distributed entries. Then $G_BK_n^B$ has the same distribution as the convex hull of a $d$-dimensional random walk $S_0:=0,S_1,\ldots,S_n$ with Gaussian increments. Similarly, $G_AK_n^A$ has the same distribution as the convex hull of a $d$-dimensional Gaussian random bridge $\widetilde{S}_0:=0,\widetilde{S}_1,\ldots,\widetilde{S}_n,\widetilde{S}_{n+1}=0$ which is essentially a Gaussian random walk conditioned on the event that it returns to $0$ in the $(n+1)$-st step. We will explain these facts in Section~\ref{section:facenumbers_walks_bridges} in more detail.


\subsection{Sums of conic intrinsic volumes in Weyl chambers and Schl\"afli orthoschemes}\label{section:angle_sums}
In this section, we state the main results of this paper concerning the sums of the conic intrinsic volumes of the tangent cones of Schl\"afli orthoschemes of type $A$ and $B$ and their products. The same is done for Weyl chambers of type $A$ and $B$ and their products.
Our first result concerning the Schl\"afli orthoschemes of types $A$ and $B$ is the following  theorem. 
\begin{theorem}\label{theorem:Schlaefli_typeB_simplex}
Let $j\in\{0,\ldots,n\}$ and $k\in\{0,\ldots,n\}$ be given. Then, it holds that
\begin{align*}
\sum_{F\in\cF_j(K^B_n)}\upsilon_k(T_{F}(K_n^B))=\sum_{F\in\cF_j(K_n^A)}\upsilon_k(T_{F}(K_n^A))=\frac{j!}{n!}\stirling{n+1}{k+1}\stirlingsec{k+1}{j+1}.
\end{align*}
\end{theorem}

As a consequence, we can derive formulas for the sums of the internal and external angles of $K_n^B$ and $K_n^A$ at their $j$-faces $F$.

\begin{corollary} For $j\in\{0,\ldots,n\}$, the sum of the internal angles is given by
\begin{align*}
\sum_{F\in\cF_j(K_n^B)}\alpha\big(T_F(K_n^B)\big)=\sum_{F\in\cF_j(K_n^A)}\alpha\big(T_F(K_n^A)\big)=\frac{j!}{n!}\stirlingsec{n+1}{j+1},
\end{align*}
while the sum of the external angles is given  by
\begin{align*}
\sum_{F\in\cF_j(K_n^B)}\alpha\big(N_F(K_n^B)\big)=\sum_{F\in\cF_j(K_n^A)}\alpha\big(N_F(K_n^A)\big)=\frac{j!}{n!}\stirling{n+1}{j+1}.
\end{align*}
\end{corollary}

\begin{proof}
The sums of the internal angles follow from Theorem~\ref{theorem:Schlaefli_typeB_simplex} with $k=n$, since $K_n^B$ and $K_n^A$ are both $n$-dimensional polytopes.
In the case of the external angles, we use that the maximal linear subspaces contained in both $T_F(K_n^B)$ and $T_F(K_n^A)$ are $j$-dimensional, which implies that   $\upsilon_j(T_F(K_n^B))= \upsilon_{n-j}((T_F(K_n^B))^\circ) = \upsilon_{n-j}(N_F(K_n^B)) = \alpha(N_F(K_n^B))$ and similarly for $K_n^A$. Using Theorem~\ref{theorem:Schlaefli_typeB_simplex} with $k=j$ completes the proof.
\end{proof}

We obtain similar results for the tangent cones of Weyl chambers of types $A$ and $B$, which are the fundamental domains of the reflection groups of the respective type; see, e.g., \cite{humphreys_book}.   More concretely, a \textit{Weyl chamber of type $B$} (or $B_{n}$) is the polyhedral cone
\begin{align*}
B^{(n)}:=\{x\in\R^n:x_1\ge x_2\ge \ldots\ge x_n\ge 0\}.
\end{align*}
The \textit{Weyl chamber of type $A$} (or $A_{n-1}$) is the polyhedral cone
\begin{align*}
A^{(n)}:=\{x\in\R^n:x_1\ge x_2\ge \ldots\ge x_n\}.
\end{align*}
We set $B^{(0)}=A^{(0)}=\{0\}$ for convenience.

\begin{theorem}\label{theorem:B-chamber}
Let $j\in\{0,\ldots,n\}$ and $k\in\{0,\ldots,n\}$ be given. Then, it holds that
\begin{align*}
\sum_{F\in\cF_j(B^{(n)})}\upsilon_k(T_{F}(B^{(n)}))
&=\frac{j!}{n!}\stirling{n+1/2}{k+1/2}_{1/2}\stirlingsec{k+1/2}{j+1/2}_{1/2}
=\frac{2^jj!}{2^nn!}\stirlingb nk \stirlingsecb kj,\\
\sum_{F\in\cF_j(A^{(n)})}\upsilon_k(T_{F}(A^{(n)}))&=\frac{j!}{n!}\stirling{n}{k}\stirlingsec{k}{j}.
\end{align*}
\end{theorem}

Theorems~\ref{theorem:Schlaefli_typeB_simplex} and~\ref{theorem:B-chamber}  are special cases of Proposition~\ref{prop:sum_intr_vol}  which we shall state in Section~\ref{section:method_of_proof} and which gives  a formula for sums of the conic intrinsic volumes of a mixed product of Weyl chambers of both types $A$ and $B$. We will give an application of  Theorems~\ref{theorem:Schlaefli_typeB_simplex} and~\ref{theorem:B-chamber} to a problem of compressed sensing in Section~\ref{section:compressed_sensing}.

It has been pointed to us by an unknown referee that there is a  similarity between Theorem~\ref{theorem:B-chamber} and the formulas, derived by Amelunxen and Lotz~\cite[Section~6.1.3]{Amelunxen2017}, for the sums of the $k$-th conic intrinsic volumes of the $j$-dimensional faces in the fans of reflection arrangements. Despite the seeming similarity between the formulas, no direct connection between the quantities under interest seems to exist, the proofs are different and, in fact, even the right-hand sides of the formulas include Stirling numbers in a different order and cannot be reduced to each other in a simple way.

For $k=n$ and $k=j$, Theorem~\ref{theorem:B-chamber} yields the following corollary on the sums of the internal and external angles of $B^{(n)}$ and $A^{(n)}$.

\begin{corollary} For $j\in\{0,\ldots,n\}$, the sums of the internal angles of $B^{(n)}$ and $A^{(n)}$ are given by
\begin{align*}
\sum_{F\in\cF_j(B^{(n)})}\alpha\big(T_F(B^{(n)})\big)=\frac{2^jj!}{2^nn!}\stirlingsecb nj,\quad\sum_{F\in\cF_j(A^{(n)})}\alpha\big(T_F(A^{(n)})\big)=\frac{j!}{n!}\stirlingsec{n}{j},
\end{align*}
while the sums of the external angles is given by
\begin{align*}
\sum_{F\in\cF_j(B^{(n)})}\alpha\big(N_F(B^{(n)})\big)=\frac{2^jj!}{2^nn!}\stirlingb nj,\quad\sum_{F\in\cF_j(A^{(n)})}\alpha\big(N_F(A^{(n)})\big)=\frac{j!}{n!}\stirling{n}{j}.
\end{align*}
\end{corollary}

\subsubsection*{Finite products of Schl\"afli orthoschemes and Weyl chambers}
The above theorems can be extended to finite products of Schl\"afli orthoschemes and Weyl chambers. Let $b\in\N$ and define $K^B:=K_{n_1}^B\times\ldots\times K_{n_b}^B$ and $K^A:=K_{n_1}^A\times\ldots\times K_{n_b}^A$ for $n_1,\ldots,n_b\in\N_0$ with $n:=n_1+\ldots+n_b$. Furthermore, for $d\in\{0,\frac{1}{2},1\}$, let
\begin{align*}
&R_d\big(k,j,b,(n_1,\ldots,n_b)\big)\\
&	\quad:=\big[t^k\big]\big[x_1^{n_1}\cdot\ldots\cdot x_b^{n_b}\big]\big[u^j\big]\frac{(1-x_1)^{-d(t+1)}\cdot\ldots\cdot (1-x_b)^{-d(t+1)}}{\big(1-u((1-x_1)^{-t}-1)\big)\cdot\ldots\cdot \big(1-u((1-x_b)^{-t}-1)\big)}.
\end{align*}
Here, $[t^N]f(t):=\frac{1}{N!}f^{(N)}(0)$ denotes the coefficient of $t^N$ in the Taylor expansion of a function $f:\R\to\R$ around $0$ and
$$
\big[x_1^{N_1}\cdot\ldots\cdot x_b^{N_b}\big]g(x_1,\ldots,x_b):=\frac{1}{N_1!\cdot\ldots\cdot N_b!}\frac{\partial^{N_1+\ldots+N_b}}{\partial x_1^{N_1}\ldots\partial x_b^{N_b}}g(0,\ldots,0)
$$
is the coefficient of $x_1^{N_1}\cdot\ldots\cdot x_b^{N_b}$ in the multidimensional Taylor expansion of a function $g:\R^b\to\R$. Note that $R_d\big(k,j,b,(n_1,\ldots,n_b)\big)=0$ for $k<j$.

\begin{theorem}\label{theorem:prod_typeB_simplex}
Let $j\in\{0,\ldots,n\}$ and $k\in\{0,\ldots,n\}$ be given. Then, it holds that
\begin{align*}
\sum_{F\in\cF_j(K^B)}\upsilon_k(T_F(K^B))=\sum_{F\in\cF_j(K^A)}\upsilon_k(T_F(K^A))=R_1\big(k,j,b,(n_1,\ldots,n_b)\big).
\end{align*}
\end{theorem}

The proof of Theorem~\ref{theorem:prod_typeB_simplex} is postponed to Section~\ref{section:proof_prod_theorems}. For finite products of Weyl chambers $W^B:=B^{(n_1)}\times\ldots\times B^{(n_b)}$ and $W^A:=A^{(n_1)}\times\ldots\times A^{(n_b)}$, we obtain the following theorems.

\begin{theorem} \label{theorem_prod_B-chambers}
For $j\in\{0,\ldots,n\}$ and $k\in\{0,\ldots,n\}$, it holds that
\begin{align*}
\sum_{F\in\cF_j(W^B)}\upsilon_k(T_F(W^B))&=R_{1/2}\big(k,j,b,(n_1,\ldots,n_b)\big),\\
\sum_{F\in\cF_j(W^A)}\upsilon_k(T_F(W^A))&=R_0\big(k,j,b,(n_1,\ldots,n_b)\big).
\end{align*}
\end{theorem}

The proof of Theorem~\ref{theorem_prod_B-chambers} is similar to that of Theorem~\ref{theorem:prod_typeB_simplex} and will be omitted. In the proof of Theorem~\ref{theorem:prod_typeB_simplex}, we will observe that if we additionally sum over all possible $n_1,\ldots,n_b$ with fixed sum $n$, the formulas in terms of Taylor coefficients simplify as follows.

\begin{proposition}\label{prop:sum_prod_simplices_A&B}
For all $j\in\{0,\ldots,n\}$ and $k\in\{0,\ldots,n\}$ we have
\begin{align*}
\sum_{\substack{n_1,\ldots,n_b\in\N_0:\\n_1+\ldots+n_b=n}}\sum_{F\in\cF_j(K^B)}\upsilon_k(T_F(K^B))
&=\sum_{\substack{n_1,\ldots,n_b\in\N_0:\\n_1+\ldots+n_b=n}}\sum_{F\in\cF_j(K^A)}\upsilon_k(T_F(K^A))\\
&=\frac{j!}{n!}\binom{j+b-1}{b-1}\stirling{n+b}{k+b}_b\stirlingsec{k+b}{j+b}_b.
\end{align*}
\end{proposition}

The proof is postponed to Section~\ref{section:proof_prop_sum_prod_a,b}.

\subsection{Method of proof of Theorems~\ref{theorem:Schlaefli_typeB_simplex} and~\ref{theorem:B-chamber}}\label{section:method_of_proof}

The main ingredient in proving Theorems~\ref{theorem:Schlaefli_typeB_simplex} and~\ref{theorem:B-chamber} is the following proposition.

\begin{proposition}\label{prop:sum_intr_vol}
Let $(j,b)\in\N_0^2\backslash\{(0,0)\}$ and $n\in\N$. For $l=(l_1,\ldots,l_{j+b})$ such that $l_1,\ldots,l_j\in\N$, $l_{j+1},\ldots,l_{j+b}\in\N_0$ and $l_1+\ldots+l_{j+b}=n$ we define
\begin{align*}
T_l:=A^{(l_1)}\times\ldots\times A^{(l_j)}\times B^{(l_{j+1})}\times\ldots\times B^{(l_{j+b})}.
\end{align*}
Then, for all $k\in\{0,\ldots,n\}$, we have
\begin{align*}
\sum_{\substack{l_1,\ldots,l_j\in\N,l_{j+1},\ldots,l_{j+b}\in\N_0:\\l_1+\ldots+l_{j+b}=n}}\upsilon_k(T_l)=\frac{j!}{n!}\stirling{n+b/2}{k+b/2}_{b/2}\stirlingsec{k+b/2}{j+b/2}_{b/2}.
\end{align*}
\end{proposition}


We will give two different ways to prove this proposition.
In Section~\ref{section:proof_prop_typeB}, we will prove it by computing the generating function of the intrinsic volumes. In Section~\ref{section:proofs_intext_angles}, we will present a different proof in which we compute the internal and external angles of the faces of the tangent cones.

%


In order to see that Theorems~\ref{theorem:Schlaefli_typeB_simplex} and~\ref{theorem:B-chamber} follow from Proposition~\ref{prop:sum_intr_vol}, we describe the collections of  tangent cones of the Schl\"afli orthoschemes and Weyl chambers of types $A$ and $B$ at their corresponding faces.

\subsubsection*{Schl\"afli orthoschemes of type $B$}
The faces of $K_n^B$ (and of any polytope in general) are obtained by replacing some of the linear inequalities in its defining conditions by equalities. Thus, each $j$-face of $K_n^B$ is determined by a collection $J:=\{i_0,\ldots,i_j\}$ of indices $0\le i_0<i_1<\ldots<i_j\le n$ and given by
\begin{multline*}
F_J:=\{x\in\R^d:1=x_1=\ldots=x_{i_0}\ge x_{i_0+1}=\ldots=x_{i_1}\ge \\\ldots\ge x_{i_{j-1}+1}=\ldots=x_{i_j}\ge x_{i_j+1}=\ldots=x_n=0\}.
\end{multline*}
Note that for $i_0=0$, no $x_i$ is required to be $1$. Similarly,  for $i_j=n$, no $x_i$ is required to be $0$. Take a point $x=(x_1,\ldots,x_n)\in\relint F_J$. For this point, all inequalities in the defining condition of $F_J$ are strict. By definition, the tangent cone of  $K_n^B$ at $F_J$ is given by
\begin{align*}
T_{F_J}(K_n^B)=\{v\in\R^n:x+\eps v\in K_n^B\text{ for some }\eps> 0\}.
\end{align*}
It follows that
\begin{multline*}
T_{F_J}(K_n^B)
	=\{v\in\R^n:0\ge v_1\ge\ldots\ge v_{i_0},v_{i_0+1}\ge\ldots\ge v_{i_1},\\
	\hspace*{5cm}\ldots,v_{i_{j-1}+1}\ge\ldots\ge v_{i_j},v_{i_j+1}\ge \ldots\ge v_n\ge 0\},
\end{multline*}
which is isometric to the product
\begin{align*}
B^{(i_0)}\times A^{(i_1-i_0)}\times\ldots\times A^{(i_j-i_{j-1})}\times B^{(n-i_j)},
\end{align*}
where the polyhedral cones
\begin{align*}
B^{(i)}:=\{x\in\R^i:x_1\ge \ldots\ge x_i\ge 0\},\quad
A^{(i)}:=\{x\in\R^i:x_1\ge\ldots\ge x_i\},\quad i\in\N_0,
\end{align*}
are the Weyl chambers of type $B$ and $A$, respectively.
We arrive at the following lemma.

\begin{lemma}\label{lemma:tangent_cones_typeB}
The collection of tangent cones $T_F(K_n^B)$, where $F$ runs through the set of all $j$-faces $\cF_j(K_n^B)$, coincides (up to isometry) with the collection
\begin{align*}
B^{(i_0)}\times A^{(i_1-i_0)}\times\ldots\times A^{(i_j-i_{j-1})}\times B^{(n-i_j)},\quad 0\le i_0<i_1<\ldots< i_j\le n.
\end{align*}
Equivalently, it coincides (up to isometry) with the collection
\begin{align*}
B^{(l_0)}\times A^{(l_1)}\times\ldots\times A^{(l_j)}\times B^{(l_{j+1})}, \quad l_0+\ldots+l_{j+1}=n, \:l_0,l_{j+1}\in\N_0,\: l_1,\ldots,l_j\in\N.
\end{align*}
If the isometry type of some cone appears with some multiplicity in one collection, then it appears with the same multiplicity in the other collections.
\end{lemma}

\subsubsection*{Schl\"afli orthoschemes of type $A$}

Now, we consider the tangent cones of Schl\"afli orthoschemes of type $A$. Recall that
\begin{align*}
K_{n}^A=\{x\in\R^{n+1}:x_1\ge\ldots\ge x_{n+1}, x_1-x_{n+1}\le 1, x_1+\ldots+x_{n+1}=0\}.
\end{align*}
Note that unlike the $B$-case, the simplex $K_n^A$ (which has dimension $n$) is contained in $\R^{n+1}$. For us it will be easier to consider the following unbounded set
\begin{align*}
\widetilde{K}_{n}^A:=\{x\in\R^{n+1}:x_1\ge\ldots\ge x_{n+1}, x_1-x_{n+1}\le 1\}.
\end{align*}
Denote by $L_{n+1}$ the $1$-dimensional linear subspace $L_{n+1}=\{x\in\R^{n+1}:x_1=\ldots=x_{n+1}\}$. Then $L_{n+1}^\perp=\{x\in\R^{n+1}:x_1+\ldots+x_{n+1}=0\}$ and we have
\begin{align*}
\widetilde{K}_n^A=L_{n+1}\oplus K_n^A,\quad K_n^A\subset L_{n+1}^\perp,
\end{align*}
where $\oplus$ denotes the orthogonal sum. Thus, there is a one-to-one correspondence $\cF_j(K_n^A)\to \cF_{j+1}(\widetilde{K}_n^A)$ between  the $j$-faces of $K_n^A$ and the $(j+1)$-faces of $\widetilde{K}_n^A$ given by $F\mapsto L_{n+1}\oplus F$. Furthermore, for every $j$-face $F$ of $K_n^A$ we have a relation between the tangent cones of $K_n^A$ and $\widetilde{K}_n^A$ given by
\begin{align}\label{eq:tangent_cones_transform}
T_{F\oplus L_{n+1}}(\widetilde{K}_n^A)=T_F(K_n^A)\oplus L_{n+1}.
\end{align}

Now, consider the collection  of tangent cones $T_F(\widetilde{K}_n^A)$, where $F\in\cF_j(\widetilde{K}_n^A)$ for some $j\in\{1,\ldots,n+1\}$ and $n\in\N_0$ more closely. The faces of $\widetilde{K}_n^A$ are obtained by replacing some inequalities in the defining conditions of $\widetilde{K}_n^A$ by equalities. Thus, there are two types of $j$-faces of $\widetilde{K}_n^A$ for $j\in\{1,\ldots,{n+1}\}$.

The $j$-faces of the first type are of the form
\begin{align}\label{eq:face_type1}
F_1=\{x\in\R^{n+1}:x_1=\ldots=x_{i_1}\ge x_{i_1+1}=\ldots=x_{i_2}\ge\ldots\ge x_{i_{j-1}+1}=\ldots=x_{n+1}, x_1-x_{n+1}\le 1\}
\end{align}
for $1\le i_1<\ldots<i_{j-1}\le n$. Note that for $j=1$, this reduces to the $1$-face $\{x\in\R^{n+1}:x_1=\ldots=x_{n+1}\}$. To determine the tangent cone at $F_1$, take some point in the relative interior of this face. For this point, all inequalities in the defining condition of $F_1$ are strict. Call this point $x=(x_1,\ldots,x_{n+1})\in\relint F_1$. By definition, we have
\begin{align*}
T_{F_1}(\widetilde{K}_n^A)=\{v\in\R^{n+1}:x+\eps v\in \widetilde{K}_n^A\text{ for some $\eps>0$}\}.
\end{align*}
It follows that
\begin{align*}
T_{F_1}(\widetilde{K}_n^A)=\{v\in\R^{n+1}: v_1\ge \ldots\ge v_{i_1},v_{i_1+1}\ge\ldots\ge v_{i_2},\ldots,v_{i_{j-1}+1}\ge \ldots\ge v_{n+1}\}.
\end{align*}
Thus, $T_{F_1}(\widetilde{K}_n^A)$ is equal to $A^{(l_1)}\times\ldots\times A^{(l_j)}$, where $l_1,\ldots,l_j\in\N$ satisfy $l_1+\ldots+l_j=n+1$ and are given by $l_1=i_1,l_2=i_2-i_1,\ldots,l_j=n+1-i_{j-1}$.

The  $j$-faces of $\widetilde{K}_n^A$ of the second type are of the form
\begin{align}\label{eq:face_type2}
F_2=\{x\in\R^{n+1}:x_1=\ldots =x_{i_1}\ge x_{i_1+1}=\ldots =x_{i_2}\ge \ldots\ge x_{i_j+1}=\ldots=x_{n+1}, x_1-x_{n+1}=1\}
\end{align}
for $1\le i_1<\ldots<i_j\le n$. The defining condition consists of $j+1$ groups of equalities and the additional condition $x_1-x_{n+1}=1$. Again, take a point $x\in\relint F_2$.
For this point all inequalities in the defining condition of $F_2$ are strict.
Hence, the tangent cone is given by
\begin{align*}
&T_{F_2}(\widetilde{K}_n^A)\\
&	\quad=\{v\in\R^{n+1}:x+\eps v\in \widetilde{K}_d^A\text{ for some $\eps>0$}\}\\
&	\quad=\{v\in\R^{n+1}:v_1\ge \ldots\ge v_{i_1},v_{i_1+1}\ge\ldots\ge v_{i_2},\ldots,v_{i_j+1}\ge \ldots\ge v_{n+1},v_1\le v_{n+1}\}\\
&	\quad=\{v\in\R^{n+1}:v_{i_1+1}\ge\ldots\ge v_{i_2},\ldots,v_{i_{j-1}+1}\ge\ldots\ge v_{i_j},v_{i_j+1}\ge\ldots\ge v_{n+1}\ge v_1\ge \ldots\ge v_{i_1}\},
\end{align*}
where in the last step we merged two groups of inequalities.
Hence, $T_{F_2}(\widetilde{K}_n^A)$ is isometric to
$
A^{(l_1+l_{j+1})}\times A^{(l_2)}\times\ldots\times A^{(l_j)},
$
where $l_1,\ldots,l_{j+1}\in\N$ are such that $l_1+\ldots+l_{j+1}=n+1$, that is they form a composition of $n+1$ into $j+1$ parts.

We can combine both types of tangent cones into one type as follows. For a $(j+1)$-composition $l_1+\ldots+l_{j+1}=n+1$, the numbers $k_1:=l_1+l_{j+1},k_2:=l_2,\ldots,k_j:=l_j$ form a $j$-composition of $n+1$. This association is not injective since each $j$-composition $k_1+\ldots+k_j=n+1$ of $n+1$ is assigned to $k_1-1$ compositions of $n+1$ into $j+1$ parts. Indeed, we can represent $k_1$ as $1+(k_1-1),2+(k_2-2),\ldots,(k_1-1)+1$. Thus, combining both types of tangent cones yields the following lemma.

\begin{lemma}\label{lemma:collection_tangentcones_A}
The collection of tangent cones $T_F(\widetilde{K}_n^A)$, where $F$ runs through the set of all $j$-faces $\cF_j(\widetilde{K}_n^A)$, coincides (up to isometry) with the collection of cones
\begin{align*}
A^{(l_1)}\times\ldots\times A^{(l_j)},\quad l_1,\ldots,l_j\in\N: l_1+\ldots+l_j=n+1,
\end{align*}
where each cone of the above collection is repeated $l_1$ times (or taken with multiplicity $l_1$).
\end{lemma}

Then, Theorem~\ref{theorem:Schlaefli_typeB_simplex} can be deduced  from Proposition~\ref{prop:sum_intr_vol} and Lemma~\ref{lemma:tangent_cones_typeB} (in the $B$-case), respectively Lemma~\ref{lemma:collection_tangentcones_A} (in the $A$-case), as follows.

\begin{proof}[Proof of Theorem~\ref{theorem:Schlaefli_typeB_simplex} assuming Proposition~\ref{prop:sum_intr_vol}]
We start with the $B$-case. For $j\in \{0,\ldots,n\}$ and $k\in\{0,\ldots,n\}$, we have
\begin{align*}
\sum_{F\in\cF_j(K_n^B)}\upsilon_k(T_F(K_n^B))
&	=\sum_{\substack{l_0,l_{j+1}\in\N_0,l_1,\ldots,l_j\in\N:\\l_0+\ldots+l_{j+1}=n}}\upsilon_k\big(B^{(l_0)}\times A^{(l_1)}\times\ldots\times A^{(l_j)}\times B^{(l_{j+1})}\big)\\
&	=\frac{j!}{n!}\stirling{n+1}{k+1}_{1}\stirlingsec{k+1}{j+1}_{1}
	=\frac{j!}{n!}\stirling{n+1}{k+1}\stirlingsec{k+1}{j+1},
\end{align*}
where we used Lemma~\ref{lemma:tangent_cones_typeB} in the first step and Proposition~\ref{prop:sum_intr_vol} with $b=2$ in the second step.

The $A$-case requires slightly more work. Using the identity $\upsilon_k(K_n^A)=\upsilon_{k+1}(K_n^A\oplus L_{n+1})$ and~\eqref{eq:tangent_cones_transform}, we obtain
\begin{align*}
\sum_{F\in\cF_j(K_n^A)}\upsilon_k(K_n^A)
&	=\sum_{F\in\cF_j(K_n^A)}\upsilon_{k+1}(T_F(K_n^A)\oplus L_{n+1})\\
&	=\sum_{F\in\cF_j(K_n^A)}\upsilon_{k+1}(T_{F\oplus L_{n+1}}(\widetilde{K}^A_n))
	=\sum_{F\in\cF_{j+1}(\widetilde{K}_n^A)}\upsilon_{k+1}(T_{F}(\widetilde{K}_n^A)).
\end{align*}
Applying Lemma~\ref{lemma:collection_tangentcones_A} $j+1$ times with multiplicity $l_1$ replaced by $l_1,\ldots,l_{j+1}$, we can observe that the collection of tangent cones $T_F(\widetilde K_n^A)$, where $F$ runs through all $(j+1)$-faces of $\cF_{j+1}(\widetilde K_n^A)$ and each cone of this collection is repeated $j+1$ times, coincides (up to isometry) with the collection of cones
\begin{align*}
A^{(l_1)}\times\ldots\times A^{(l_{j+1})},\quad l_1,\ldots,l_{j+1}\in\N: l_1+\ldots+l_{j+1}=n+1,
\end{align*}
where each cone is taken with multiplicity $l_1+\ldots+l_{j+1}= n+1$. Therefore, we arrive at
\begin{align*}
\sum_{F\in\cF_{j+1}(\widetilde K_n^A)}(j+1)\upsilon_{k+1}(T_F(\widetilde K_n^A))
&	=\sum_{\substack{l_1,\ldots,l_{j+1}\in\N:\\ l_1+\ldots+l_{j+1}=n+1}}
(n+1) \upsilon_{k+1}\big(A^{(l_1)}\times\ldots\times A^{(l_{j+1})}\big)\\
&	=(n+1)\frac{(j+1)!}{(n+1)!}\stirling{n+1}{k+1}_0\stirlingsec{k+1}{j+1}_0\\
&	=\frac{(j+1)!}{n!}\stirling{n+1}{k+1}\stirlingsec{k+1}{j+1},
\end{align*}
where we used Proposition~\ref{prop:sum_intr_vol}. Dividing both sides by $j+1$ yields the claim.
\end{proof}

\subsubsection*{Weyl chambers of type $B$}

For $n\in\N$ recall that
\begin{align*}
B^{(n)}:=\{x\in\R^n:x_1\ge x_2\ge\ldots\ge x_n\ge 0\}
\end{align*}
and $B^{(0)}:=\{0\}$ by convention. Now, let $j\in\{0,\ldots,n\}$. Each $j$-face of $B^{(n)}$ is determined by a collection $J:=\{i_1,\ldots,i_j\}$ of indices $1\le i_1<\ldots<i_j\le n$, and given by
\begin{align*}
F_J:=\{x\in\R^n:x_1=\ldots=x_{i_1}\ge \ldots\ge x_{i_{j-1}+1}=\ldots=x_{i_j}\ge x_{i_j+1}=\ldots=x_{n}=0\}.
\end{align*}
Note that for $i_j=n$, no $x_i$'s are required to be $0$, and for $j=0$, we obtain the $0$-dimensional face $\{0\}$. In order to determine the tangent cone $T_{F_J}(B^{(n)})$ take a point $x=(x_1,\ldots,x_n)\in\relint F_J$. Again, this point satisfies the defining conditions of $F_J$ with inequalities replaced by strict inequalities. Thus, the tangent cone is given by
\begin{align*}
T_{F_J}(B^{(n)})
&	=\{v\in\R^n:x+\eps v\in B^{(n)}\text{ for some $\eps>0$}\}\\
&	=\{v\in\R^n:v_1\ge \ldots\ge v_{i_1},\ldots,v_{i_{j-1}+1}\ge\ldots\ge v_{i_j},v_{i_j+1}\ge \ldots\ge v_n\ge 0\}\\
&	=A^{(i_1)}\times A^{(i_2-i_1)}\times\ldots\times A^{(i_j-i_{j-1})}\times B^{(n-i_j)}.
\end{align*}
The above reasoning yields the following lemma.

\begin{lemma}\label{lemma:B-chamber_tangent_cones}
The collection of tangent cones $T_F(B^{(n)})$, where $F$ runs through the set of all $j$-faces $\cF_j(B^{(n)})$, coincides  with the collection of polyhedral cones
\begin{align*}
A^{(l_1)}\times\ldots\times A^{(l_j)}\times B^{(l_{j+1})},\quad l_1+\ldots+l_{j+1}=n,\: l_1,\ldots,l_j\in\N,\: l_{j+1}\in\N_0.
\end{align*}
\end{lemma}

\subsubsection*{Weyl chambers of type $A$}

For $n\in\N$ recall that
\begin{align*}
A^{(n)}:=\{x\in\R^n:x_1\ge x_2\ge\ldots\ge x_n\}.
\end{align*}
For  $j\in\{1,\ldots,n\}$ every $j$-face of $A^{(n)}$ is determined by a collection $J:=\{i_1,\ldots,i_{j-1}\}$ of indices $1\le i_1<\ldots<i_{j-1}\le n-1$, and given by
\begin{align*}
F_J:=\{x\in\R^n:x_1=\ldots=x_{i_1}\ge \ldots\ge x_{i_{j-1}+1}=\ldots=x_{n}\}.
\end{align*}
Note that for $j=1$, we obtain the $1$-face $\{x\in\R^n:x_1=\ldots=x_n\}$. In order to determine the tangent cone $T_{F_J}(A^{(n)})$ consider a point $x=(x_1,\ldots,x_n)\in\relint F_J$. In a fashion similar to the case of a $B$-type Weyl chamber, we can characterize the tangent cone of $A^{(n)}$ at $F_J$ as follows:
\begin{align*}
T_{F_J}(A^{(n)})
&	=\{v\in\R^n:x+\eps v\in A^{(n)}\text{ for some $\eps>0$}\}\\
&	=\{v\in\R^n:v_1\ge \ldots\ge v_{i_1},\ldots,v_{i_{j-1}+1}\ge\ldots\ge v_{n}\}\\
&	=A^{(i_1)}\times A^{(i_2-i_1)}\times\ldots\times A^{(n-i_{j-1})}.
\end{align*}
This yields the following analogue of Lemma~\ref{lemma:B-chamber_tangent_cones}.

\begin{lemma}\label{lemma:A-chamber_tangent_cones}
The collection of tangent cones $T_F(A^{(n)})$, where $F$ runs through the set of all $j$-faces $\cF_j(A^{(n)})$, coincides  with the collection of polyhedral cones
\begin{align*}
A^{(l_1)}\times\ldots\times A^{(l_j)},\quad l_1+\ldots+l_{j}=n,\: l_1,\ldots,l_j\in\N.
\end{align*}
\end{lemma}

\begin{proof}[Proof of Theorem~\ref{theorem:B-chamber} assuming Proposition~\ref{prop:sum_intr_vol}]
We start with the $B$-case. For $j\in\{0,\ldots,n\}$ and $k\in\{0,\ldots,n\}$, we have
\begin{align*}
\sum_{F\in\cF_j(B^{(n)})}\upsilon_k\big(T_F(B^{(n)})\big)
&	=\sum_{\substack{l_1,\ldots,l_j\in\N,l_{j+1}\in\N_0:\\l_1+\ldots+l_{j+1}=n}}\upsilon_k\big(A^{(l_1)}\times\ldots\times A^{(l_j)}\times B^{(l_{j+1})}\big)\\
&	=\frac{j!}{n!}\stirling{n+1/2}{k+1/2}_{1/2}\stirlingsec{k+1/2}{j+1/2}_{1/2}
	=\frac{j!}{n!}2^{j-n}\stirlingb nk \stirlingsecb nk,
\end{align*}
where we used Lemma~\ref{lemma:B-chamber_tangent_cones} in the first step, Proposition~\ref{prop:sum_intr_vol} with $b=1$ in the second step and the formulas in~\eqref{eq:rel_1/2-Stirling_B-analogues} in the last step.

In the $A$-case, for $j\in\{1,\ldots,n\}$ and $k\in\{1,\ldots,n\}$, we have
\begin{align*}
\sum_{F\in\cF_j(A^{(n)})}\upsilon_k\big(T_F(A^{(n)})\big)
	=\sum_{\substack{l_1,\ldots,l_j\in\N:\\l_1+\ldots+l_{j}=n}}\upsilon_k\big(A^{(l_1)}\times\ldots\times A^{(l_j)}\big)
	=\frac{j!}{n!}\stirling{n}{k}_0\stirlingsec{k}{j}_{0}
	=\frac{j!}{n!}\stirling{n}{k}\stirlingsec{k}{j},
\end{align*}
where we used Lemma~\ref{lemma:A-chamber_tangent_cones} in the first step and Proposition~\ref{prop:sum_intr_vol} with $b=0$ in the second step. For $j=0$ or $k=0$ the formula is evidently true as well.
\end{proof}

\subsubsection*{Identities for the generalized Stirling numbers}
Let us also mention that Proposition~\ref{prop:sum_intr_vol} yields the following identities
for the generalized Stirling numbers. 

\begin{corollary} \label{cor:stirlingnumbers_sum}
For $n\in\N$,   $j\in\{0,\ldots,n\}$ and $2b\in \N_0$, the following identities hold:
\begin{align}
&\sum_{k=0}^n\stirling{n+b}{k+b}_{b}\stirlingsec{k+b}{j+b}_b=\frac{n!}{j!}\binom{n+2b-1}{j+2b-1}=\frac{n!}{j!}\binom{n+2b-1}{n-j}
\label{eq:sum_prod_r-stirling},\\
&\sum_{k=0}^n(-1)^k\stirling{n+b}{k+b}_b\stirlingsec{k+b}{j+b}_b=0.
\label{eq:altern_sum_prod_stirling}
\end{align}
\end{corollary}
Note that~\eqref{eq:altern_sum_prod_stirling} is a special case of the orthogonality relation between the $r$-Stirling numbers proven by Broder~\cite[Theorem 25]{Broder_1982}. Relation~\eqref{eq:sum_prod_r-stirling} is also known for $b=0,1$, the numbers on the right-hand side being the Lah numbers. 

\begin{proof} 
We use the properties~\eqref{eq:intrinsic_sum1} and~\eqref{eq:intrinsic_sum2}
for the conic intrinsic volumes of a cone $C$ that is not a linear subspace. With the notation for $T_l$ introduced in Proposition~\ref{prop:sum_intr_vol} and with $b$ replaced by $2b$, this yields
\begin{align*}
\sum_{k=0}^n\:\sum_{\substack{l_1,\ldots,l_j\in\N,l_{j+1},\ldots,l_{j+2b}\in\N_0:\\l_1+\ldots+l_{j+2b}=n}}\upsilon_k(T_l)
=
\sum_{\substack{l_1',\ldots,l_{j+2b}'\in\N:\\l_1'+\ldots+l_{j+2b}'=n+2b}}1=\binom{n+2b-1}{j+2b-1}=\binom{n+2b-1}{n-j},
\end{align*}
where we used the well-known fact that the number of compositions of $n$ into $k$ positive integers is $\binom{n-1}{k-1}$. Similarly, we obtain
\begin{align*}
\sum_{k=0}^n(-1)^k\sum_{\substack{l_1,\ldots,l_j\in\N,l_{j+1},\ldots,l_{j+2b}\in\N_0:\\l_1+\ldots+l_{j + 2b}=n}}\upsilon_k(T_l)=0.
\end{align*}
On the other hand, applying Proposition~\ref{prop:sum_intr_vol} with $b$ replaced by $2b$ yields
\begin{align*}
\sum_{k=0}^n\sum_{\substack{l_1,\ldots,l_j\in\N,l_{j+1},\ldots,l_{j+2b}\in\N_0:\\l_1+\ldots+l_{j+2b}=n}}\upsilon_k(T_l)=\sum_{k=0}^n\frac{j!}{n!}\stirling{n+b}{k+b}_b\stirlingsec{k+b}{j+b}_b
\end{align*}
and
\begin{align*}
\sum_{k=0}^n(-1)^k\sum_{\substack{l_1,\ldots,l_j\in\N,l_{j+1},\ldots,l_{j+2b}\in\N_0:\\l_1+\ldots+l_{j+2b}=n}}\upsilon_k(T_l)
=
\sum_{k=0}^n(-1)^k\frac{j!}{n!}\stirling{n+b}{k+b}_b\stirlingsec{k+b}{j+b}_b,
\end{align*}
which proves~\eqref{eq:sum_prod_r-stirling} and~\eqref{eq:altern_sum_prod_stirling}.
\end{proof}

\subsection{Expected face numbers: Convex hulls of Gaussian random walks and bridges}\label{section:facenumbers_walks_bridges}

The above theorems on the angle sums of the tangent cones of Schl\"afli orthoschemes yield results on the expected number of faces of Gaussian random walks and random bridges and their Minkowski sums. Consider independent $d$-dimensional standard Gaussian random vectors
$$
X_1^{(1)},X_2^{(1)},\ldots,X_{n_1}^{(1)},X_1^{(2)},X_2^{(2)},\ldots,X_{n_2}^{(2)},\ldots, X_{1}^{(b)},X_2^{(b)},\ldots,X_{n_b}^{(b)}
$$ and let $n_1+\ldots+n_b=n\geq d$.
For every $i\in \{1,\ldots,b\}$  we define a random walk $(S_0^{(i)},S_1^{(i)},\ldots,S_{n_i}^{(i)})$  by
\begin{align*}
S_k^{(i)}=X_1^{(i)}+\ldots+X_k^{(i)},\quad k=1,\ldots,n_i, \quad S_0^{(i)}:=0.
\end{align*}
Consider the convex hulls of these random walks:
\begin{align*}
C^{(i)}_{n_i}:=\conv\big\{S_0^{(i)},S_1^{(i)},\ldots,S_{n_i}^{(i)}\big\},\quad i=1,\ldots,b.
\end{align*}
The following theorem gives a formula for the expected number of $j$-faces of the Minkowski sum of $b$ such convex hulls defined by
$$
C^{(1)}_{n_1}+\ldots+C^{(b)}_{n_b}=\{v_1+\ldots+v_b:v_1\in C^{(1)}_{n_1},\ldots,v_b\in C^{(b)}_{n_b}\}.
$$

\begin{theorem}\label{theorem:sum_walks}
Let $0\le j<d\le n$ be given and define $C^{(1)}_{n_1},\ldots,C^{(b)}_{n_b}$ as above. Then, we have
\begin{align*}
\E f_j(C^{(1)}_{n_1}+\ldots+C^{(b)}_{n_b})
&	=2\sum_{l\ge 1}R_1(d-2l+1,j,b,(n_1,\ldots,n_b)),
\end{align*}
where we recall that
\begin{align*}
&R_1(k,j,b,(n_1,\ldots,n_b))\\
&	\quad:=\big[t^k\big]\big[x_1^{n_1}\cdot\ldots\cdot x_b^{n_b}\big]\big[u^j\big]\frac{(1-x_1)^{-(t+1)}\cdot\ldots\cdot (1-x_b)^{-(t+1)}}{\big(1-u((1-x_1)^{-t}-1)\big)\cdot\ldots\cdot \big(1-u((1-x_b)^{-t}-1)\big)}.
\end{align*}
\end{theorem}

The same formula holds for the Minkowski sum of Gaussian random bridges which are essentially Gaussian random walks under the condition that they return to $0$ in the last step. To state it, consider independent $d$-dimensional standard Gaussian random vectors
$$
X_1^{(1)},X_2^{(1)},\ldots,X_{n_1+1}^{(1)},X_1^{(2)},X_2^{(2)},\ldots,X_{n_2+1}^{(2)},\ldots, X_{1}^{(b)},X_2^{(b)},\ldots,X_{n_b+1}^{(b)}
$$
and define the random walks $({S}^{(i)}_0,{S}^{(i)}_1,\ldots,{S}^{(i)}_{n_i+1})$ as above.
We define the Gaussian bridge $(\widetilde{S}^{(i)}_0,\widetilde{S}^{(i)}_1,\ldots,\widetilde{S}^{(i)}_{n_i+1})$ as a process having  the same distribution as the Gaussian random walk $(S_0^{(i)},S_1^{(i)},\ldots,S_{n_i+1}^{(i)})$ conditioned on the event that $S^{(i)}_{n_i+1}=0$. Equivalently, the Gaussian bridge $(\widetilde{S}^{(i)}_0,\widetilde{S}^{(i)}_1,\ldots,\widetilde{S}^{(i)}_{n_i+1})$ can be constructed as
\begin{align*}
\widetilde{S}_k^{(i)}:=S^{(i)}_k-\frac{k}{n_i+1}S^{(i)}_{n_i+1},\quad k=1,\ldots,{n_i+1}
\end{align*}
for $i=1,\ldots,b$. Note that $\widetilde{S}^{(i)}_{n_i+1}=0$. Define the convex hulls of the Gaussian bridges  by
\begin{align*}
\widetilde{C}^{(i)}_{n_i}:=\conv\big\{\widetilde{S}^{(i)}_0,\widetilde{S}^{(i)}_1,\ldots,\widetilde{S}^{(i)}_{n_i}\big\},\quad i=1,\ldots,b.
\end{align*}

\begin{theorem}\label{theorem:sum_brides}
Let $0\le j<d\le n$ be given and $\widetilde{C}^{(1)}_{n_1},\ldots,\widetilde{C}^{(a)}_{n_b}$ as above. Then, we have
\begin{align*}
\E f_j(\widetilde{C}^{(1)}_{n_1}+\ldots+\widetilde{C}^{(b)}_{n_b})
&	=2\sum_{l\ge 1}R_1(d-2l+1,j,b,(n_1,\ldots,n_b)).
\end{align*}
\end{theorem}

For a single convex hull ($b=1$), the expected number of $j$-faces of $C^{(1)}_n$ and $\widetilde{C}^{(1)}_{n}$ is already known (in a more general case), see~\cite[Theorems 1.2 and 5.1]{KVZ17}, and given by  
\begin{align*}
\E f_j(C^{(1)}_n)=\E f_j(\widetilde{C}_{n}^{(1)})=\frac{2\cdot j!}{n!}\sum_{l=0}^\infty\stirling{n+1}{d-2l}\stirlingsec{d-2l}{j+1}.
\end{align*}
This formula is a special case of  Theorems~\ref{theorem:sum_walks} and~\ref{theorem:sum_brides} with $b=1$ and $n_1$ replaced by $n$.

%
%
%
%
\subsection{Method of proof of Theorems~\ref{theorem:sum_walks} and~\ref{theorem:sum_brides}}\label{section:methodprood_walksbridges}

The main ingredient in the proof of the named theorems is the following lemma which is due to Affentranger and Schneider~\cite[(5)]{AS92}.

\begin{lemma}\label{lemma:facenumber_polytopes}
Let $P\subset \R^n$ be a convex polytope with non-empty interior and $G\in\R^{d\times n}$ be a Gaussian random matrix, that is, its entries are independent and standard normal random variables. Then, for all $0\le j<d\le n$ we have
\begin{align*}
\E f_j(G P)
&	=f_j(P)-\sum_{F\in F_j(P)}\gamma_d\big(T_F(P)\big),
\end{align*}
where the Grassmann angles $\gamma_d$ were defined in~\eqref{eq:grassmann_angle_def}.
\end{lemma}
In fact, Affentranger and Schneider~\cite{AS92} proved this formula for a random orthogonal projection of $P$, while the fact that Gaussian matrices yield the same result follows from a result of Baryshnikov and Vitale~\cite{BV94}.
Due to the relation between Grassmann angles and conic intrinsic volumes stated in~\eqref{eq:relation_grass_intr}, the lemma can be written as
\begin{align*}
\E f_j(G P)=f_j(P)-2\sum_{F\in F_j(P)}\sum_{l=0}^\infty\upsilon_{d+2l+1}\big(T_F(P)\big).
\end{align*}
Using~\eqref{eq:sum_odd_Intr_vol}, it also follows that under the conditions of Lemma~\ref{lemma:facenumber_polytopes}  we have 
\begin{align}\label{eq:proj_lemma_alt2}
\E f_j(G P)=2\sum_{F\in F_j(P)}\big(\upsilon_{d-1}\big(T_F(P)\big)+\upsilon_{d-3}\big(T_F(P)\big)+\ldots\big).
\end{align}

Now take a Gaussian matrix $G_B=(\xi_{i,j})\in\R^{d\times n}$, where $\xi_{i,j}$, $i\in\{1,\ldots,d\},j\in\{1,\ldots,n\}$ are independent standard Gaussian random variables. Then, we claim that $G_BK^B$ has the same distribution as the Minkowski $C^{(1)}_{n_1}+\ldots+C^{(b)}_{n_b}$.  Similarly, for a Gaussian matrix $G_A\in\R^{d\times (n+b)}$, we claim that $G_AK^A$ has the same distribution as $\widetilde{C}^{(1)}_{n_1}+\ldots+\widetilde{C}^{(b)}_{n_b}$.

In order to see this, consider the case of a single Schl\"afli orthoscheme $K_{n_1}^B$  first. Let $G_B^{(1)}=(\xi_{i,j})\in\R^{d\times n_1}$ be a Gaussian matrix. We know that the Schl\"afli orthoscheme $K_{n_1}^B$ is the simplex given as the  convex hull of the $n_1$-dimensional vectors
$$
(0,0,0,\ldots,0),(1,0,0\ldots,0),(1,1,0,\ldots,0),\ldots,(1,1,1\ldots,1).
$$
Thus, we obtain
\begin{align*}
G_B^{(1)} K_{n_1}^B
&	=\conv\big\{G_B^{(1)}(0,\ldots,0)^\top,G_B^{(1)}(1,0,\ldots,0)^\top,G_B^{(1)}(1,1,0,\ldots,0)^\top,\ldots,G_B^{(1)}(1,\ldots,1)^\top\big\}\\
&	=\conv\left\{\begin{pmatrix}0 \\ \vdots \\ 0 \end{pmatrix} ,\begin{pmatrix}\xi_{1,1} \\ \vdots \\ \xi_{d,1} \end{pmatrix},\begin{pmatrix}\xi_{1,1}+\xi_{1,2} \\ \vdots \\ \xi_{d,1}+\xi_{d,2} \end{pmatrix},\ldots,\begin{pmatrix}\xi_{1,1}+\ldots+\xi_{1,n_1} \\ \vdots \\ \xi_{d,1}+\ldots+\xi_{d,n_1} \end{pmatrix}				\right\}\\
&	\eqdistr \conv\big\{S_0^{(1)},S_1^{(1)},\ldots,S_{n_1}^{(1)}\big\}\\
&	=C_{n_1}^{(1)},
\end{align*}
where $\eqdistr$ denotes the distributional equality of two random elements.

Similarly, we consider a Gaussian matrix $G_A^{(1)}\in\R^{d\times (n_1+1)}$ for the Schl\"afli orthoscheme $K_{n_1}^A$. We know that $K_{n_1}^A$ is the convex hull of the $(n_1+1)$-dimensional vectors $P_0=(0,0,\ldots,0)$ and
\begin{align*}
P_i=(\underbrace{1,\ldots,1}_{i},0,\ldots,0)-\frac{i}{n_1+1}(1,1,\ldots,1),\quad 1\le i\le n_1.
\end{align*}
Thus, in the same way we get
\begin{align*}
G_A^{(1)} K_{n_1}^A
&	=\conv\{G_A^{(1)}P_0,G_A^{(1)}P_1,\ldots,G_A^{(1)}P_{n_1}\}\\
&	\eqdistr\conv\big\{0,S_1^{(1)}-\frac{1}{n_1+1}S_{n_1+1}^{(1)},S_{2}^{(1)}-\frac{2}{n_1+1}S_{n_1+1}^{(1)},\ldots,S_{n_1}^{(1)}-\frac{n_1}{n_1+1}S_{n_1+1}^{(1)}\big\}\\
&	\eqdistr\conv\big\{\widetilde{S}^{(1)}_0,\widetilde{S}^{(1)}_1,\widetilde{S}^{(1)}_2,\ldots,\widetilde{S}^{(1)}_{n_1}\big\}\\
&	=\widetilde{C}^{(1)}_{n_1}.
\end{align*}

The product case follows from the following observation. Let $G_B\in\R^{d\times n}$ be a Gaussian matrix and $n_1+\ldots+n_b=n$. Then, we can represent $G_B$ as the row of  $b$ independent matrices $G_B=(G_B^{(1)},\ldots,G_B^{(b)})$, where $G_B^{(i)}\in\R^{d\times n_i}$ is itself a Gaussian matrix for $i=1,\ldots,b$. We can represent each vector $x\in\R^n$ as the column of $b$ vectors, i.e.~$x=(x^{(1)},\ldots,x^{(b)})^\top$, where $x^{(i)}\in\R^{n_i}$ for $i=1,\ldots,b$. Then, we easily observe that
\begin{align*}
G_Bx=(G_B^{(1)},\ldots,G_B^{(b)})(x^{(1)},\ldots,x^{(b)})^\top=G_B^{(1)}x^{(1)}+\ldots+G_B^{(b)}x^{(b)}.
\end{align*}
It follows that
\begin{align*}
G_BK^B
	=(G_B^{(1)},\ldots,G_B^{(b)})(K_{n_1}^B\times\ldots\times K_{n_b}^{(b)})
	=G_B^{(1)}K_{n_1}^B+\ldots+G_B^{(b)}K_{n_b}^B
	\eqdistr C^{(1)}_{n_1}+\ldots+C^{(b)}_{n_b}.
\end{align*}
In the same way we obtain for a Gaussian matrix $G_A\in\R^{d\times (n+b)}$ with $n_1+\ldots+n_b=n$ that
\begin{align*}
G_AK^A\eqdistr \widetilde{C}^{(1)}_{n_1}+\ldots+\widetilde{C}^{(b)}_{n_b}.
\end{align*}
Now, we can finally prove  Theorems~\ref{theorem:sum_walks} and~\ref{theorem:sum_brides}.

\begin{proof}[Proof of Theorem~\ref{theorem:sum_walks}]
Let $1\le j<d\le n$ and $G_B\in\R^{d\times n}$ be a Gaussian matrix. As we already observed,  $G_BK^B\eqdistr C^{(1)}_{n_1}+\ldots+C^{(b)}_{n_b}$. Thus,~\eqref{eq:proj_lemma_alt2} yields
%
%
\begin{align*}
\E f_j(C^{(1)}_{n_1}+\ldots+C^{(b)}_{n_b})
&	=\E f_j(G_BK^B)\\
&	=2\sum_{F\in \cF_j (K^B)} (\upsilon_{d-1}(T_F(K^B))+\upsilon_{d-3}(T_F(K^B))+\ldots)\\
&	=2\sum_{l\ge 1}R_1(d-2l+1,j,b,(n_1,\ldots,n_b)),
\end{align*}
where we used Theorem~\ref{theorem:prod_typeB_simplex} in the last step.
\end{proof}

\begin{proof}[Proof of Theorem~\ref{theorem:sum_brides}]
Let $1\le j<d\le n$ and let $G_A\in\R^{d\times (n+b)}$ be a Gaussian matrix. We have already seen that $G_AK^A\eqdistr \widetilde{C}^{(1)}_{n_1}+\ldots+\widetilde{C}^{(b)}_{n_b}$. Although the polytope $K^A\subset\R^{n+b}$ is only $n$-dimensional, we can still apply Lemma~\ref{lemma:facenumber_polytopes}, and therefore also~\eqref{eq:proj_lemma_alt2}, to the ambient linear subspace $\lin K^A$ since the Grassmann angles do not depend on the dimension of the ambient linear subspace. Combining this with Theorem~\ref{theorem:prod_typeB_simplex}, we obtain
\begin{align*}
\E f_j(\widetilde{C}^{(1)}_{n_1}+\ldots+\widetilde{C}^{(b)}_{n_b})
&	=\E f_j(G_AK^A)\\
&	=2\sum_{F\in \cF_j (K^B)} (\upsilon_{d-1}(T_F(K^A))+\upsilon_{d-3}(T_F(K^A))+\ldots)\\
&	=2\sum_{l\ge 1}R_1(d-2l+1,j,b,(n_1,\ldots,n_b)),
\end{align*}
which completes the proof.
\end{proof}

\subsection{Application to compressed sensing}\label{section:compressed_sensing}
Let us briefly mention an application of Theorems~\ref{theorem:Schlaefli_typeB_simplex} and~\ref{theorem:B-chamber} to compressed sensing. Donoho and Tanner~\cite{donoho_tanner_orthants} have considered the following problem. Let $x=(x_1,\ldots,x_n)$ be an unknown signal belonding to some  polyhedral set $P\subset \R^n$ and let $G:\R^n\to \R^k$ be a Gaussian matrix, where $k\leq n$. We would like to recover the signal $x$ from its image $y=Gx$. Following Donoho and Tanner~\cite{donoho_tanner_orthants}, we denote the event that such a recovery is uniquely possible by
\begin{align*}
\text{Unique}(G,x,P): \quad\text{The system $y=Gx'$ has a unique solution $x'=x$ in $P$}.
\end{align*}
The  recovery is uniquely possible if and only if $(x+\ker G) \cap P = \{x\}$.
Assume that $x\in \relint F$ for some $j$-face $F\in\cF_j(P)$ of $P$. Then,
the following equivanelce holds:
\begin{align}\label{eq:equivalence_unique}
\text{Unique}(G,x,P)\Leftrightarrow\ker G\cap T_F(P)=\{0\}.
\end{align}
Using this observation, Donoho and Tanner~\cite{donoho_tanner_orthants} computed the probability of unique recovery explicitly in the cases when $P=\R_+^n$ is the non-negative orthant or $P=[0,1]^n$ is the unit cube. By the way, the equivalence \eqref{eq:equivalence_unique} was stated by Donoho and Tanner~\cite[Lemmas 2.1,5.1]{donoho_tanner_orthants} in these special cases, but it easily generalizes to any polyhedral set.
We are now going to use Theorems~\ref{theorem:Schlaefli_typeB_simplex} and~\ref{theorem:B-chamber} to compute the probabilities of unique signal recovery in the case when $P$ is a Weyl chamber or a Schl\"afli orthoscheme, which corresponds to natural isotonic constraints frequently imposed in statistics.

\subsubsection*{Weyl chambers}
We consider the following model for a random signal $x$ that belongs to the Weyl chamber $B^{(n)}$ of type $B$. Let $0 <  j \le n$ be given together with $j$ positive numbers $a_1,\ldots, a_j$.  Select a random and uniform subset $\{i_1,\ldots,i_j\}\subseteq\{1,\ldots,n\}$ with $1\le i_1<\ldots<i_j\le n$. Now, define the corresponding $j$-face $F^B(i_1,\ldots,i_j)$ of $B^{(n)}$ by
$$
F^B(i_1,\ldots,i_j):=\{x\in\R^n:x_1=\ldots=x_{i_1}\ge \ldots\ge x_{i_{j-1}+1}=\ldots=x_{i_j}\ge x_{i_j+1}=\ldots=x_n=0\}
$$
and consider the signal $x=(x_1,\ldots,x_n)$ given by
\begin{align*}
x_m=\sum_{l:\,i_l\ge m} a_l, \quad m=1,\ldots,n.
\end{align*}
Then, by construction, $F^B(i_1,\ldots,i_j)$ is random and uniformly distributed on $\cF_j(B^{(n)})$. Moreover,  $x$ belongs to $\relint F^B(i_1,\ldots,i_j)$. 

\begin{proposition}\label{prop:unique_weyl_B}
Let $0  \leq  j \leq  k \le n$ and let $x\in\R^n$ be a random signal constructed as above (where we put $x=0$ if $j=0$). If $G:\R^n\to \R^k$ is a Gaussian random matrix, then it holds that
\begin{align*}
\bP\big[\textup{Unique}\big(G,x,B^{(n)}\big)\big]=
\frac{2^{j+1}j!}{2^nn!\binom nj}\sum_{i=1,3,5,\ldots}\stirlingb n{k - i}\stirlingsecb{k-i}j.
\end{align*}
\end{proposition}

\begin{proof}
Using the above equivalence~\eqref{eq:equivalence_unique} and the construction of $x$, we obtain that
\begin{align*}
\bP\big[\textup{Unique}\big(G,x,B^{(n)}\big)\big]
&	=\frac{1}{\binom nj}\sum_{1\le i_1<\ldots<i_j\le n}\bP\big[\ker G\cap T_{F^B(i_1,\ldots,i_j)}(B^{(n)})=\{0\}\big].
\end{align*}
Since $\ker G$ is rotationally invariant, and thus, a uniformly distributed $(n-k)$-dimensional linear subspace, we conclude that the probability on the right-hand side coincides with $1 - \gamma_k(T_F(B^{(n)}))$, where $\gamma_k(T_F(B^{(n)}))$ denotes the $k$-th Grassmann angle  of $T_F(B^{(n)})$. Thus, relation~\eqref{eq:relation_grass_intr}  yields
\begin{align*}
\bP\big[\textup{Unique}\big(G,x,B^{(n)}\big)\big]
&	=\frac{1}{\binom nj}\sum_{F\in \cF_j(B^{(n)})} (1 - \gamma_k\big(T_F(B^{(n)})\big))\\
&	=\frac{2}{\binom nj} \sum_{i=1,3,5,\ldots}\sum_{F\in \cF_j(B^{(n)})}\upsilon_{k-i}\big(T_F(B^{(n)})\big)\\
&   =\frac{2^{j+1}j!}{2^nn!\binom nj}\sum_{i=1,3,5,\ldots}\stirlingb n{k - i}\stirlingsecb{k-i}j,
\end{align*}
where the last equality follows from Theorem~\ref{theorem:B-chamber}.
\end{proof}

In the $A$-case, the construction is analogous. Let $0 <  j \le n$ and let $a_1,\ldots,a_{j}$ be positive numbers. Select a random and uniform subset $\{i_1,\ldots,i_{j-1}\}\subseteq \{1,\ldots,n-1\}$ with $1\le i_1<\ldots<i_{j-1}\le n-1$, put $i_j:=n$,  and define the corresponding $j$-face $F^A(i_1,\ldots,i_{j-1})$ of $A^{(n)}$ by
\begin{align*}
F^A(i_1,\ldots,i_{j-1}):=\{x\in\R^n:x_1=\ldots=x_{i_1}\ge \ldots\ge x_{i_{j-1}+1}=\ldots=x_n\},
\end{align*}
which is uniformly distributed in $\cF_j(A^{(n)})$.
We define the random signal $x=(x_1,\ldots,x_n)$ by $x_m=\sum_{l:i_l\ge m} a_l$, for $m=1,\ldots,n$. Then, $x$  belongs to the relative interior of $F^A(i_1,\ldots,i_{j-1})$. This yields the following proposition, which is analogous to the $B$-case and can be proven in a similar way.

\begin{proposition}
Let $0  <  j \leq  k \le n$ and let $x\in\R^n$ be a random signal constructed as above. If $G:\R^n\to \R^k$ is a Gaussian random matrix, then it holds that
\begin{align*}
\bP\big[\textup{Unique}\big(G,x,A^{(n)}\big)\big]=\frac{2\cdot j!}{n!\binom {n-1}{j-1}}\sum_{i=1,3,5,\ldots}\stirling {n}{k-i}\stirlingsec{k-i}{j}.
\end{align*}
\end{proposition}

\subsubsection*{Schl\"afli orthoschemes}
Again, we start with the $B$-case. Let $0\le j \le n$ be given.
Take $j+1$ positive numbers $a_0,a_1,\ldots, a_j$ satisfying $a_0+\ldots+a_j=1$  and select a random and uniform subset $\{i_0,\ldots,i_j\}\subseteq\{0,\ldots,n\}$ with $0\le i_0<\ldots<i_j\le n$. Now, define  the corresponding $j$-face $S^B(i_0,\ldots,i_j)$ of $K_n^B$ by
\begin{multline*}
S^B(i_0,\ldots,i_j):=\{x\in\R^n:1=x_1=\ldots=x_{i_0}\ge x_{i_0+1}=\ldots=x_{i_1}\ge  \ldots\\\ge x_{i_{j-1}+1}=\ldots=x_{i_j}\ge x_{i_j+1}=\ldots=x_n=0\}.
\end{multline*}
Consider the signal $x=(x_1,\ldots,x_n)$ given by
\begin{align*}
x_m=\sum_{l:\,i_l\ge m} a_l,\quad m=1,\ldots,n.
\end{align*}
By construction, we have $x\in\relint S^B(i_0,\ldots,i_j)$ and $S^B(i_0,\ldots,i_j)$ is random and uniformly distributed on $\cF_j(K^B_n)$. Then, the $B$-case of Theorem~\ref{theorem:Schlaefli_typeB_simplex} yields the following proposition.

\begin{proposition}
Let $0  \leq   j \leq  k \le n$ and let $x\in\R^n$ be a random signal constructed as above. If $G:\R^n\to \R^k$ is a Gaussian random matrix, then it holds that
\begin{align*}
\bP\big[\textup{Unique}\big(G,x,K^B_n\big)\big] = \frac{2\cdot j!}{n!\binom {n+1}{j+1}}\sum_{i=0,2,4,\ldots}\stirling {n+1}{k - i}\stirlingsec{k- i }{j+1}.
\end{align*}
\end{proposition}

In the $A$-case, recall the definition of the Schl\"afli orthoscheme of type $A$:
\begin{align*}
K_n^A=\{x\in\R^{n+1}:x_1\ge\ldots\ge x_{n+1}, x_1-x_{n+1}\le 1,x_1+\ldots+x_{n+1}=0\}.
\end{align*}
Take $0\leq j\leq n$ and  positive numbers $a_1,\ldots,a_{j+1}$ satisfying $a_1+\ldots+a_{j+1}=1$. Also, select a random and uniform subset $\{i_1,\ldots,i_{j+1}\}\subseteq \{1,\ldots,n+1\}$ such that $1\le i_1<\ldots<i_{j+1}\le n+1$. If $i_{j+1}=n+1$, we define the corresponding $j$-face of $K_n^A$ to be
\begin{multline*}
S^A(i_1,\ldots,i_{j},n+1):=\{x\in\R^{n+1}:x_1=\ldots=x_{i_1}\ge x_{i_1+1}=\ldots=x_{i_2}\ge\ldots\\\ge x_{i_{j}+1}=\ldots=x_{n+1}, x_1-x_{n+1}\le 1,x_1+\ldots+x_{n+1}=0\}
\end{multline*}
(which corresponds to the type of face described in~\eqref{eq:face_type1}), while for $i_j\le n$, we define
\begin{multline*}
S^A(i_1,\ldots,i_{j+1}):=\{x\in\R^{n+1}:x_1=\ldots=x_{i_1}\ge x_{i_1+1}=\ldots=x_{i_2}\ge\ldots\\\ge x_{i_{j+1}+1}=\ldots=x_{n+1}, x_1-x_{n+1}= 1,x_1+\ldots+x_{n+1}=0\}
\end{multline*}
(which corresponds to the type of face described in~\eqref{eq:face_type2}). Due to the one-to-one correspondence between the collections of indices $1\le i_1<\ldots< i_{j+1}\le n+1$ and the $j$-faces of $K_n^A$, the face $S^A(i_1,\ldots,i_{j+1})$ is uniformly distributed on the set $\cF_j(K_n^A)$ of all $j$-faces. We consider the random signal $x=(x_1,\ldots,x_{n+1})$ given by
\begin{align*}
x_m=\bigg(\sum_{l:\,i_l\ge m} a_l\bigg)-c,\quad m=1,\ldots,n+1,
\end{align*}
where $c$ is chosen to ensure that  $x_1+\ldots+x_{n+1} = 0$.
It can be easily checked that the signal $x$ belongs to the relative interior of $S^A(i_1,\ldots,i_j)$. Then, the $A$-case of Theorem~\ref{theorem:Schlaefli_typeB_simplex} yields the following proposition.

\begin{proposition}
Let $0  \leq   j \leq  k \le n$ and let $x\in\R^{n+1}$ be a random signal constructed as above. If $G:\R^{n+1}\to \R^k$ is a Gaussian random matrix, then it holds that
\begin{align*}
\bP\big[\textup{Unique}\big(G,x,K^A_n\big)\big]=\frac{2\cdot j!}{n!\binom {n+1}{j+1}}\sum_{i=0,2,4,\ldots}\stirling {n+1}{k - i}\stirlingsec{k- i }{j+1}.
\end{align*}
\end{proposition}

\section{Proofs: Angle sums of Weyl chambers and Schl\"afli orthoschemes}\label{section:proofs_gen_fct}

In this section, we present the proofs of Propositions~\ref{prop:sum_intr_vol}, \ref{prop:sum_prod_simplices_A&B}, and Theorem~\ref{theorem:prod_typeB_simplex}.
Most of the proofs rely on  explicit expressions for the conic intrinsic volumes of the Weyl chambers.
Recall that the Weyl chambers $A^{(n)}$ and $B^{(n)}$ are defined by
\begin{align}
A^{(n)}:=\{x\in\R^n:x_1\ge x_2\ge\ldots\ge x_n\},
\;\;\;
B^{(n)}:=\{x\in\R^{n}: x_1\ge x_2\ge\ldots\ge x_n\ge 0\},
\end{align}
for $n\in\N$,
and we put $B^{(0)}:=\{0\}$. The intrinsic volumes of the Weyl chambers are known explicitly and given by
\begin{align}\label{eq_in_vol_Weyl_chamb}
\nu_k(A^{(n)})=\stirling{n}{k}\frac{1}{n!},\quad \nu_k(B^{(n)})=\frac{\stirlingb nk}{2^{n}n!}
\end{align}
for $k=0,\ldots,n$; see, e.g., \cite[Theorem 4.2]{KVZ15} or \cite[Theorem 1.1]{GK20_Intr_Vol}.  The $\stirlingb nk$'s denote the $B$-analogues of the Stirling numbers of the first kind as defined in~\eqref{eq:def_stirling1b}.

\subsection{Proof of Proposition~\ref{prop:sum_intr_vol}}\label{section:proof_prop_typeB}
Let $(j,b)\in\N_0^2\backslash\{(0,0)\}$. Recall that for $l=(l_1,\ldots,l_{j+b})$ such that $l_1,\ldots,l_j\in\N$, $l_{j+1},\ldots,l_{j+b}\in\N_0$ and $l_1+\ldots+l_{j+b}=n$, we define
\begin{align*}
T_l=A^{(l_1)}\times\ldots\times A^{(l_j)}\times B^{(l_{j+1})}\times\ldots\times B^{(l_{j+b})}.
\end{align*}
Our goal is to show that  for all $k\in\{0,\ldots,n\}$,
\begin{align}\label{eq:to_show_prop_Bsimplex}
\sum_{\substack{l_1,\ldots,l_j\in\N,l_{j+1},\ldots,l_{j+b}\in\N_0:\\l_1+\ldots+l_{j+b}=n}}\upsilon_k(T_l)=\frac{j!}{n!}\stirling{n+b/2}{k+b/2}_{b/2}\stirlingsec{k+b/2}{j+b/2}_{b/2}.
\end{align}

\begin{proof}[Proof of Proposition~\ref{prop:sum_intr_vol}]
Let $(j,b)\in\N_0^2\backslash\{(0,0)\}$ and $k\in\{0,\ldots,n\}$ be given. By the product formula for conic intrinsic volumes~\eqref{eq:intr_vol_prod_gen_fct},
the generating polynomial of the intrinsic volumes of $T_l$ can be written as
\begin{align*}
P_{T_l}(t):=\sum_{k=0}^n\upsilon_k(T_{l})t^k
&=\left(\sum_{m=0}^{l_1}\upsilon_m(A^{(l_1)})t^m\right)\cdot\ldots\cdot\left(\sum_{m=0}^{l_j}\upsilon_m(A^{(l_j)})t^m\right)\\ &\times\left(\sum_{m=0}^{l_{j+1}}\upsilon_m(B^{(l_{j+1})})t^m\right)\cdot\ldots\cdot\left(\sum_{m=0}^{l_{j+b}}\upsilon_m(B^{(l_{j+b})})t^m\right).
\end{align*}
We can consider each sum on the right-hand side separately. Using the representations of the Stirling numbers of the first kind and their $B$-analogues from~\eqref{eq:def_stirling1_polynomial} and~\eqref{eq:def_stirling1b}, as well as the intrinsic volumes of the Weyl chambers stated in~\eqref{eq_in_vol_Weyl_chamb}, we obtain
\begin{align}\label{eq:identity_A_intrinsic_volumes}
\sum_{m=0}^{i}\upsilon_m(A^{(i)})t^m=\frac{1}{i!}\sum_{m=0}^{i}\stirling{i}{n}t^m=\frac{1}{i!}t(t+1)\cdot\ldots\cdot(t+i-1),
\qquad i\in\N,
\end{align}
and
\begin{align*}
\sum_{m=0}^{i}\upsilon_m(B^{(i)})t^m
&	=\frac{1}{2^{i}i!}\sum_{m=0}^{i}\stirlingb im t^m=\frac{1}{2^{i}i!}(t+1)(t+3)\cdot\ldots\cdot(t+2i-1),
\qquad i\in\N_0.
\end{align*}
Note that for $i=0$, we put $(t+1)(t+3)\cdot\ldots\cdot(t+2i-1):=1$ by convention, which is consistent with $\upsilon_0(\{0\})=1$. This yields the following formula
\begin{align*}
P_{T_l}(t)
&	=\frac{(t+1)(t+3)\cdot\ldots\cdot(t+2l_{j+1}-1)}{2^{l_{j+1}}l_{j+1}!}\cdot\ldots\cdot\frac{(t+1)(t+3)\cdot\ldots\cdot(t+2l_{j+b}-1)}{2^{l_{j+b}}l_{j+b}!}\\
&	\quad\times\frac{t^{\overline{l_1}}}{l_1!}\cdot\ldots\cdot\frac{t^{\overline{l_j}}}{l_j!},
\end{align*}
where $t^{\overline{r}}:=t(t+1)\cdot\ldots\cdot (t+r-1)$, $r\in\N$,  denotes   the rising factorial. Thus, the $k$-th conic intrinsic volume of $T_l$ is the coefficient of $t^k$ in the above polynomial $P_{T_l}(t)$. Note that this already implies $\nu_k(T_l)=0$ for $k<j$. Thus, the left-hand side of~\eqref{eq:to_show_prop_Bsimplex} is $0$, which coincides with the right-hand side, since for $k<j$ we have
$$
\stirlingsec{k+b/2}{j+b/2}_{b/2}=0.
$$

Therefore, we only need to consider the case $k\ge j$. Let $P^{(n)}_{l_1,\ldots,l_{j+b}}(m)$, where $m=0,\ldots,n$, be the coefficients of the polynomial
\begin{multline*}
t^{\overline{l_1}}\cdot\ldots\cdot t^{\overline{l_{j}}}(t+1)(t+3)\cdot\ldots\cdot(t+2l_{j+1}-1)(t+1)(t+3)\cdot\ldots\cdot(t+2l_{j+b}-1)=\sum_{m=j}^n P^{(n)}_{l_1,\ldots,l_{j+b}}(m)t^m.
\end{multline*}
Using the notation just introduced,  we obtain
\begin{align*}
&\sum_{\substack{l_1,\ldots,l_j\in\N,l_{j+1},\ldots,l_{j+b}\in\N_0:\\l_1+\ldots+l_{j+b}=n}}\upsilon_k(T_l)	\quad=\sum_{\substack{l_1,\ldots,l_j\in\N,l_{j+1},\ldots,l_{j+b}\in\N_0:\\l_1+\ldots+l_{j+b}=n}}\frac{P^{(n)}_{l_1,\ldots,l_{j+b}}(k)}{l_1!\cdot\ldots\cdot l_{j+b}!2^{l_{j+1}+\ldots+l_{j+b}}}.
\end{align*}

Now, let $[t^N]f(t):=\frac{1}{N!}f^{(N)}(0)$ be the coefficient of $t^N$ in the Taylor expansion of a function $f$ around $0$. Define
\begin{multline*}
C_{n,j,b}(k):=[t^k]\sum_{\substack{l_1,\ldots,l_j\in\N,l_{j+1},\ldots,l_{j+b}\in\N_0:\\l_1+\ldots+l_{j+b}=n}}\Bigg(\frac{t^{\overline{l_1}}}{l_1!}\cdot\ldots\cdot\frac{t^{\overline{l_j}}}{l_j!}\\
	\times\frac{(t+1)(t+3)\cdot\ldots\cdot
(t+2l_{j+1}-1)}{2^{l_{j+1}}l_{j+1}!}\cdot\ldots\cdot\frac{(t+1)(t+3)\cdot\ldots\cdot(t+2l_{j+b}-1)}{2^{l_{j+b}}l_{j+b}!}\Bigg).
\end{multline*}
Then, we can observe that
\begin{align*}
&C_{n,j,b}(k)\\
&	=\sum_{\substack{l_1,\ldots,l_j\in\N,l_{j+1},\ldots,l_{j+b}\in\N_0:\\l_1+\ldots+l_{j+b}=n}}\frac{1}{l_1!\cdot\ldots\cdot l_{j+b}!2^{l_{j+1}+\ldots+l_{j+b}}}[t^k]\Big((t+1)(t+3)\cdot\ldots\cdot(t+2l_{j+1}-1)\\
&	\quad\times\ldots\cdot(t+1)(t+3)\cdot\ldots\cdot(t+2l_{j+b}-1)t^{\overline{l_1}}\cdot\ldots\cdot t^{\overline{l_j}}\Big)\\
&	=\sum_{\substack{l_1,\ldots,l_j\in\N,l_{j+1},\ldots,l_{j+b}\in\N_0:\\l_1+\ldots+l_{j+b}=n}}\frac{P^{(n)}_{l_1,\ldots,l_{j+b}}(k)}{l_1!\cdot\ldots\cdot l_{j+b}!2^{l_{j+1}+\ldots+l_{j+b}}}\\
&	=\sum_{\substack{l_1,\ldots,l_j\in\N,l_{j+1},\ldots,l_{j+b}\in\N_0:\\l_1+\ldots+l_{j+b}=n}}\upsilon_k(T_l).
\end{align*}
Thus, to prove the proposition, it suffices to show that for all $(j,b)\in\N_0^2\backslash\{(0,0)\}$ and $k\in\{j,\ldots,n\}$ we have
\begin{align}\label{eq_to_prove}
C_{n,j,b}(k)=\frac{j!}{n!}\stirling{n+b/2}{k+b/2}_{b/2}\stirlingsec{k+b/2}{j+b/2}_{b/2}.
\end{align}
To this end,  we can introduce a new variable $x$ and write, by expanding the product,
\begin{align*}
C_{n,j,b}(k)=[t^k][x^n]\Bigg(\bigg(\sum_{l=0}^\infty\frac{(t+1)(t+3)\cdot\ldots\cdot(t+2l-1)}{2^ll!}x^l\bigg)^b\times\bigg(\sum_{l=1}^\infty\frac{t^{\overline{l}}}{l!}x^l\bigg)^j\Bigg).
\end{align*}
Using~\eqref{eq:def_stirling1_polynomial} and the exponential generating function in two variables for the Stirling numbers of the first kind stated in~\eqref{eq:gen_fct_stirling1}, we obtain
\begin{align}\label{eq:sum/l!}
\sum_{l=1}^\infty \frac{t^{\overline{l}}}{l!}x^l=-1+\sum_{l=0}^\infty\sum_{m=0}^l\stirling{l}{m}t^m\frac{x^l}{l!}=(1-x)^{-t}-1.
\end{align}
From~\eqref{eq:def_stirling1b} and~\eqref{eq:gen_fct_stirling1b}, we similarly get
\begin{align}\label{eq:sum_B}
\sum_{l=0}^\infty\frac{(t+1)(t+3)\cdot\ldots\cdot(t+2l-1)}{2^ll!}x^l
&	=\sum_{l=0}^\infty\sum_{m=0}^l\stirlingb lm t^m\Big(\frac{x}{2}\Big)^l\frac{1}{l!}=(1-x)^{-\frac{1}{2}(t+1)}.
\end{align}
Thus, we have
\begin{align*}
C_{n,j,b}(k)
	=[t^k][x^n]\Big((1-x)^{-\frac{b}{2}(t+1)}\big((1-x)^{-t}-1\big)^j\Big)
	=[x^n][t^k]\big(e^{-\frac{cb}{2}(t+1)}(e^{-ct}-1)^j\big),
\end{align*}
where we set $c=c(x)=\log(1-x)$. The exponential generating function of the $b/2$-Stirling numbers stated in~\eqref{eq:gen_fct_r-stirling2} yields
\begin{align*}
e^{-\frac{cb}{2}t}(e^{-ct}-1)^j=\sum_{m=j}^\infty\stirlingsec{m+b/2}{j+b/2}_{b/2}\frac{j!}{m!}(-ct)^m.
\end{align*}
It follows that
\begin{align*}
[t^k]\big(e^{-\frac{cb}{2}(t+1)}(e^{-ct}-1)^j\big)
	=e^{-\frac{cb}{2}}[t^k]\bigg(\sum_{m=j}^\infty\stirlingsec{m+b/2}{j+b/2}_{b/2}\frac{j!}{m!}(-ct)^m\bigg) =e^{-\frac{cb}{2}}(-c)^k\frac{j!}{k!}\stirlingsec{k+b/2}{j+b/2}_{b/2}.
\end{align*}
Furthermore, using~\eqref{eq:gen_fct_r-stirling1} we obtain
\begin{align*}
[x^n]\big(e^{-\frac{cb}{2}}(-c)^k\big)
&	=[x^n]\frac{\big(-\log(1-x)\big)^k}{(1-x)^{\frac{b}{2}}}
	=[x^n]\bigg(\sum_{m=k}^\infty \stirling{m+b/2}{k+b/2}_{b/2}\frac{k!}{m!}x^m\bigg)
	=\frac{k!}{n!}\stirling{n+b/2}{k+b/2}_{b/2}.
\end{align*}
Taking all this into consideration, we obtain
\begin{align*}
C_{n,j,b}(k)
&	=[x^n][t^k]\big(e^{-\frac{cb}{2}(t+1)}(e^{-ct}-1)^j\big)\\
&	=[x^n]\bigg(e^{-\frac{cb}{2}}(-c)^k\frac{j!}{k!}\stirlingsec{k+b/2}{j+b/2}_{b/2}\bigg)\\
&	=\frac{j!}{k!}\stirlingsec{k+b/2}{j+b/2}_{b/2}[x^n]\big(e^{-\frac{cb}{2}}(-c)^k\big)\\
&	=\frac{j!}{n!}\stirling{n+b/2}{k+b/2}_{b/2}\stirlingsec{k+b/2}{j+b/2}_{b/2},
\end{align*}
which coincides with (\ref{eq_to_prove}) and therefore completes the proof.
\end{proof}

\subsection{Alternative proof of Proposition~\ref{prop:sum_intr_vol}}\label{section:proofs_intext_angles}

In this section, we present a more direct way to prove Proposition~\ref{prop:sum_intr_vol} by computing the internal and external angles of the faces of the Weyl chambers. A similar approach was already used in~\cite{GV01} to compute the classical intrinsic volumes of the Schl\"afli orthoscheme of type $B$ and in~\cite{GK20_Intr_Vol} to compute the conic intrinsic volumes of the Weyl chambers.

Let $(j,b)\in\N_0^2\backslash\{(0,0)\}$. Recall that for $l=(l_1,\dots,l_{j+b})$ with $l_1,\dots,l_j\in\N$, $l_{j+1},\dots,l_{j+b}\in\N_0$ and $l_1+\ldots+l_{j+b}=n$, we have
\begin{align*}
T_l=A^{(l_1)}\times\ldots\times A^{(l_j)}\times B^{(l_{j+1})}\times\ldots\times B^{(l_{j+b})}.
\end{align*}

\begin{proof}[Alternative proof of Proposition~\ref{prop:sum_intr_vol}]
Let 
$k\in\{0,\dots,n\}$ be given. Using~\eqref{eq:in_vol_internalexternal}, we obtain
\begin{align*}
\sum_{\substack{l_1,\dots,l_j\in\N,l_{j+1},\dots,l_{j+b}\in\N_0:\\l_1+\ldots+l_{j+b}=n}}\upsilon_k(T_l)
=\sum_{\substack{l_1,\dots,l_j\in\N,l_{j+1},\dots,l_{j+b}\in\N_0:\\l_1+\ldots+l_{j+b}=n}}\sum_{F\in\cF_k(T_l)}\alpha(F)\alpha(N_F(T_l)).
\end{align*}
In order to compute this sum, we consider the $k$-faces of $T_l$. Since $T_l$ is the product of Weyl chambers of type $A$ and type $B$, its $k$-faces are products of faces of the individual Weyl chambers. The $r$-faces of the Weyl chamber  $A^{(d)}$ can be represented by
\begin{align*}
A^{(d)}(i_1,\dots,i_{r-1}):=\{x\in\R^{d}:x_1=\ldots=x_{i_1}\ge\ldots\ge x_{i_{r-1}}=\ldots=x_{d}\}\in\cF_r(A^{(d)})
\end{align*}
for $d\in\N$, $1\le r\le d$ and $1\le i_1<\ldots<i_{r-1}\le d-1$. Note that for $r=1$, we obtain the $1$-face $\{x\in\R^d:x_1=\ldots=x_d\}$. Thus, $T_l$ has no $k$-faces for $k<j$ since the faces of $A^{(l_1)},\ldots, A^{(l_j)}$ have at least dimension $1$. Therefore, for $k<j$ the left-hand side of~\eqref{eq:to_show_prop_Bsimplex} is $0$ coinciding with the right-hand side. From now on assume $k\ge j$. The $s$-faces of the Weyl chamber  $B^{(d)}$ are given by
\begin{multline*}
B^{(d)}(m_1,\dots,m_{s}):=\{x\in\R^d:x_1=\ldots=x_{m_1}\ge \ldots\ge x_{m_{s-1}+1}=\ldots=x_{m_s}\ge\\ x_{m_s+1}=\ldots=x_d=0\}\in\cF_s(B^{(d)})
\end{multline*}
for $d\in\N$, $0\le s\le d$ and $1\le m_1<\ldots<m_s\le d$. Note that for $m_s=d$, no $x_i$'s are required to be $0$ and for $s=0$, we obtain the $0$-dimensional face $\{0\}$.

Thus, for each $k$-face $F$ of $T_l$, there is a collection of indices $r_1,\dots,r_j\in \N$, $s_1,\dots,s_b\in \N_0$ satisfying $r_1+\ldots+r_j+s_1+\ldots+s_b=k$, such that the face can be written as
\begin{multline}\label{eq:faces_TJ}
F_{l,i,m}:=A^{(l_1)}(i_1^{(1)},\dots,i^{(1)}_{r_1-1})\times\ldots\times A^{(l_j)}(i_1^{(j)},\dots,i^{(j)}_{r_j-1}) \\\times B^{(l_{j+1})}(m_1^{(1)},\dots,m_{s_1}^{(1)})\times\ldots\times B^{(l_{j+b})}(m_1^{(b)},\dots,m_{s_b}^{(b)})
\end{multline}
for the vectors $i=(i_1^{(1)},\dots,i^{(1)}_{r_1-1},\dots,i_1^{(j)},\dots,i^{(j)}_{r_j-1})$ and $m=(m_1^{(1)},\dots,m^{(1)}_{s_1},\dots,m_1^{(b)},\dots,m^{(b)}_{s_b})$ satisfying the conditions
\begin{align}\label{eq:cond_l1}
&1\le i_1^{(1)}<\ldots<i^{(1)}_{r_1-1}\le l_1-1,\dots,1\le i_1^{(j)}<\ldots<i^{(j)}_{r_j-1}\le l_j-1,\\
&1\le m_1^{(1)}<\ldots<m^{(1)}_{s_1}\le l_{j+1},\dots,1\le m_1^{(b)}<\ldots<m^{(b)}_{s_b}\le l_{j+b}\notag.
\end{align}
For every $r_1,\dots,r_j,s_1,\dots,s_b$ as above and each $(i,m)$ satisfying conditions~\eqref{eq:cond_l1}, the cone $F_{l,i,m}$ defines a $k$-face of $T_J$.

\subsubsection*{Internal angles} Now, we want to compute the solid angle of $F_{l,i,m}$. Since $F_{l,i,m}$ is given as a product of faces, the angle $\alpha(F_{l,i,m})$ is also given as a product of angles. Thus, it is enough to evaluate the angles of $A^{(d)}(i_1,\dots,i_{r-1})$ and $B^{(d)}(m_1,\dots,m_{s})$. Consider the linear hulls of $A^{(d)}(i_1,\dots,i_{r-1})$ and $B^{(d)}(m_1,\dots,m_{s})$ given by
\begin{align*}
\lin A^{(d)}(i_1,\dots,i_{r-1})=\{x\in\R^{i}:x_1=\ldots=x_{i_1},\ldots, x_{i_{r-1}+1}=\ldots=x_{d}\}
\end{align*}
and
\begin{align*}
\lin B^{(d)}(m_1,\dots,m_{s})=\{x\in\R^d:x_1=\ldots=x_{m_1}, \ldots,x_{m_{s-1}+1}=\ldots=x_{m_s}, x_{m_s+1}=\ldots=x_d=0\}.
\end{align*}
Thus, the vectors $y_1,\dots,y_r$ given by
\begin{multline*}
y_1=\frac{1}{\sqrt{i_1}}(\overbrace{1,\dots,1}^{1,\dots,i_1},0,\dots,0),
y_2=\frac{1}{\sqrt{i_2-i_1}}(0,\dots,0,\overbrace{1,\dots,1}^{i_1+1,\dots,i_2},0,\dots,0),\\\dots, y_{r}=\frac{1}{\sqrt{d-i_{r-1}}}(0,\dots,0,\overbrace{1,\dots,1}^{i_{r-1}+1,\dots,d})
\end{multline*}
form an orthonormal basis of $\lin A^{(d)}(i_1,\dots,i_{r-1})$. Similarly, the vectors $z_1,\dots,z_s$ form an orthonormal basis of $\lin B^{(d)}(m_1,\dots,m_{s})$, where
\begin{multline*}
z_1=\frac{1}{\sqrt{m_1}}(\overbrace{1,\dots,1}^{1,\dots,m_1},0,\dots,0),z_2=\frac{1}{\sqrt{m_2-m_1}}(0,\dots,0,\overbrace{1,\dots,1}^{m_1+1,\dots,m_2},0,\dots,0),\\\dots, z_s=\frac{1}{\sqrt{m_{s}-m_{s-1}}}(0,\dots,0,\overbrace{1,\dots,1}^{m_{s-1}+1,\dots,m_s},0,\dots,0).
\end{multline*}
Now let $\xi_1,\xi_2,\dots$ be independent and standard normal distributed random variables. Then, the random vector $N_1:=\xi_1y_1+\ldots+\xi_ry_r$ is $r$-dimensional standard normal distributed on the linear subspace $\lin A^{(d)}(i_1,\dots,i_{r-1})$ and the random vector $N_2:=\xi_1z_1+\ldots+\xi_sz_s$ is $s$-dimensional standard normal distributed on the linear subspace $\lin B^{(d)}(m_1,\dots,m_{s})$. Thus, by definition, we obtain
\begin{align*}
\alpha\big(A^{(d)}(i_1,\dots,i_{r-1})\big)
&	=\P(N_1\in A^{(d)}(i_1,\dots,i_{r-1}))\\
&	=\P\Big(\frac{\xi_1}{\sqrt{i_1}}\ge\frac{\xi_2}{\sqrt{i_2-i_1}}\ge\ldots\ge \frac{\xi_r}{\sqrt{d-i_{r-1}}}\Big)
\end{align*}
and similarly
\begin{align*}
\alpha\big(B^{(d)}(m_1,\dots,m_{s})\big)
&	=\P(N_2\in B^{(d)}(m_1,\dots,m_{s}))
\\
&	=\P\Big(\frac{\xi_1}{\sqrt{m_1}}\ge\frac{\xi_2}{\sqrt{m_2-m_1}}\ge\ldots\ge \frac{\xi_{s}}{\sqrt{m_{s}-m_{s-1}}}\ge 0\Big).
\end{align*}
We will not evaluate these probabilities explicitly.
In the course of this proof,  we will divide them into groups
so that each group of these probabilities adds up to $1$.
In summary, the solid angle of $F_{l,i,m}$ is given by
\begin{multline}\label{eq:int_angles}
\alpha(F_{l,i,m})=\alpha\big(A^{(l_1)}(i_1^{(1)},\dots,i^{(1)}_{r_1-1})\big)\cdot\ldots\cdot\alpha\big( A^{(l_j)}(i_1^{(j)},\dots,i^{(j)}_{r_j-1})\big) \\\times \alpha\big(B^{(l_{j+1})}(m_1^{(1)},\dots,m_{s_1}^{(1)})\big)\cdot\ldots\cdot \alpha\big(B^{(l_{j+b})}(m_1^{(b)},\dots,m_{s_b}^{(b)})\big).
\end{multline}

\subsubsection*{External angles} Now, we need to evaluate the external angle of $F_{l,i,m}$, i.e.~$\alpha(N_{F_{l,i,m}}(T_l))$. In order to do this, we consider the normal cone $N_{F_{l,i,m}}(T_l)=(\lin F_{l,i,m})^\perp\cap (T_l)^\circ$. We have
\begin{align*}
(T_l)^\circ=\big(A^{(l_1)}\big)^\circ\times\ldots\times \big(A^{(l_j)} \big)^\circ\times \big(B^{(l_{j+1})}\big)^\circ\times\ldots\times \big(B^{(l_{j+b})}\big)^\circ,
\end{align*}
where for example
\begin{align*}
\big(A^{(l_1)}\big)^\circ
&	=\{x\in\R^{l_1}:x_1\ge\ldots\ge x_{l_1}\}^\circ\\
&	=\pos\{e_2-e_1,e_3-e_2,\dots,e_{l_1}-e_{l_1-1}\}\\
&	=\{x\in\R^{l_1}: x_1\le 0,x_1+x_2\le 0,\dots,x_1+\ldots+x_{l_1-1}\le 0, x_1+\ldots+x_{l_1}=0\}.
\end{align*}
using the duality relation~\eqref{eq:duality_relation} in the second line. Here, $e_1,\dots,e_{l_1}$ denotes the standard Euclidean orthonormal basis in the ambient space $\R^{l_1}$. Thus, the partial sums of the coordinates of the points from the dual cones of Weyl chambers of type $A$ form ``bridges'' staying below zero. Similarly, we obtain
\begin{align*}
\big(B^{(l_{j+1})}\big)^\circ
&	=\pos\{e_2-e_1,e_3-e_2,\dots,e_{l_{j+1}}-e_{l_{j+1}-1},-e_{l_{j+1}}\}\\
&	=\{x\in\R^{l_{j+1}}:x_1\le 0,x_1+x_2\le 0,\dots,x_1+\ldots+x_{l_{j+1}}\le 0\},
\end{align*}
which we will also refer to as a ``walk'' staying below zero. The dual cones of the other Weyl chambers are obtained in the same way.

Recalling the definition of $F_{l,i,m}$ in~\eqref{eq:faces_TJ}, its linear hull is given as the product of linear hulls of the faces $A^{(l_1)}(i_1^{(1)},\dots,i^{(1)}_{r_1-1}),\dots, B^{(l_{j+b})}(m_1^{(b)},\dots,m_{s_b}^{(b)})$ which we already evaluated. Thus, $(\lin F_{l,i,m})^\perp$ is the product of the orthogonal complements of the linear hulls. For the $r_1$-face $A^{(l_1)}(i_1^{(1)},\dots,i^{(1)}_{r_1-1})$, we have
\begin{align*}
\big(\lin A^{(l_1)}(i_1^{(1)},\dots,l^{(1)}_{i_1-1})\big)^\perp
&	=\{x\in\R^{l_1}:x_1=\ldots=x_{i^{(1)}_1},\ldots, x_{i^{(1)}_{r_1-1}}=\ldots=x_{l_1}\}^\perp\\
&	=\{x\in\R^{l_1}:x_1+\ldots+x_{i^{(1)}_1}=0,\dots,x_{i^{(1)}_{r_1-1}}+\ldots+x_{l_1}=0\}.
\end{align*}
Similarly, for $B^{(l_{j+1})}(m_1^{(1)},\dots,m^{(1)}_{s_1})$, we get
\begin{align*}
\big(\lin B^{(l_{j+1})}(m_1^{(1)},\dots,m^{(1)}_{s_1})\big)^\perp
&	=\{x\in\R^{l_{j+1}}:x_1+\ldots+x_{m^{(1)}_1}=0,\ldots,x_{m^{(1)}_{s_1-1}+1}+\ldots+x_{m^{(1)}_{s_1}}=0\}.
\end{align*}
Thus, $(\lin F_{l,i,m})^\perp$ consists of the points $x= (x^{(1)},\ldots,x^{(j+b)}) \in\R^n$ with
$x^{(1)}\in \R^{l_1}, \ldots x^{(j+b)}\in \R^{l_{j+b}}$ whose components $x^{(q)}= (x^{(q)}_1,\ldots, x^{(q)}_{l_q})$ are such that the following conditions are satisfied:
$$
x_1^{(q)}+\ldots+x_{i^{(q)}_1}^{(q)}=0,\dots,x_{i^{(q)}_{r_q-1}}^{(q)}+\ldots+x_{l_q}^{(q)}=0
$$
for all $q\in \{1,\ldots,j\}$ and
$$
x_1^{(j+d)}+\ldots+x_{m^{(d)}_1}^{(j+d)}=0,\ldots,x_{m^{(d)}_{s_{d}-1}+1}^{(j+d)}+\ldots+x_{m^{(d)}_{s_{d}}}^{(j+d)}=0
$$
for all $d\in \{1,\ldots,b\}$. In other words, each $A$-face corresponds to a group of $l_q$ coordinates  which is divided into smaller groups whose sums are required to be $0$.   Similarly, each $B$-face corresponds to a group of $l_{j+d}$ coordinates which is divided into smaller subgroups whose sums are required to be $0$, except for the last subgroup.

Taking both the description of the defining conditions of $(T_l)^\circ$ and $(\lin F_{l,i,m})^\perp$ into consideration, we obtain for the normal cone $N_{F_{l,i,m}}(T_l)=(T_l)^\circ\cap (\lin F_{l,i,m})^\perp$ the following conditions: $N_{F_{l,i,m}}(T_l)$ consists of the points $x= (x^{(1)},\ldots,x^{(j+b)}) \in\R^n$ with
$x^{(1)}\in \R^{l_1}, \ldots x^{(j+b)}\in \R^{l_{j+b}}$, such that  the partial sums of the groups
\begin{align*}
(x_1^{(q)},\ldots,x_{i^{(q)}_1}^{(q)}),\dots,(x_{i^{(q)}_{r_q-1}}^{(q)},\ldots,x_{l_q}^{(q)})
\end{align*}
for all $q\in\{1,\dots,j\}$ and
\begin{align*}
(x_1^{(j+d)},\ldots,x_{m^{(d)}_1}^{(j+d)}),\ldots,(x_{m^{(d)}_{s_{d}-1}+1}^{(j+d)},\ldots,x_{m^{(d)}_{s_{d}}}^{(j+d)}),
\end{align*}
for all $d\in \{1,\ldots,b\}$, form  bridges staying non-positive, while  the partial sums of the groups
\begin{align*}
(x_{m^{(1)}_{s_1}}^{(j+1)},\dots,x_{l_{j+1}}^{(j+1)}),\dots,(x_{m^{(b)}_b}^{(j+b)},\dots,x_{l_{j+b}}^{(j+b)}).
\end{align*}
form walks staying non-positive. Note that each $A$-face corresponds to a collection of bridges, while each $B$-face corresponds to a collection of bridges and one walk.

%
%
The angles of the cones defined by such bridge or walk conditions can be interpreted as the probabilities that the corresponding bridges or walks stay non-positive. These probabilities are known explicitly thanks to the works of Sparre Andersen~\cite{Andersen1949}, \cite{sparre_andersen1}; see also~\cite{GV01}, \cite{KVZ15}, \cite[Section 3.1]{GK20_Intr_Vol} for works relating these formulas to convex cones.  The angle of a cone defined by a bridge condition is given by
$$
\alpha(\{x\in\R^i:x_1\le 0,\dots,x_1+\ldots+x_{i-1}\le 0,x_1+\ldots+x_i=0\}) = 1/i.
$$
Furthermore, the angle of a cone  $W:=\{x\in\R^i:x_1\le 0,x_1+x_2\le 0,\dots,x_1+\ldots+ x_i\le 0\}$ defined by a walk condition is given by
\begin{align*}
\alpha(W)
&	=\P(N\in W)\\
&	=\P(N_1\le 0,N_1+N_2\le 0,\dots,N_1+\ldots+N_i\le 0)\\
&	=\binom{2i}{i}\frac{1}{2^{2i}}
\end{align*}
by a formula due to Sparre Andersen~\cite{Andersen1949}, where $N=(N_1,\dots,N_i)$ is an $i$-dimensional standard normal distributed vector. Taking all into consideration, we have
\begin{align}\label{eq:ext_angles}
&\alpha(N_{F_{l,i,m}}(T_l))\nonumber\\
&	=\frac{1}{i_1^{(1)}(i_2^{(1)}-i_1^{(1)})\cdot\ldots\cdot(l_1-l_{r_1-1}^{(1)})\cdots i_1^{(j)}(i_2^{(j)}-i_1^{(j)})\cdot\ldots\cdot(l_j-i_{r_j-1}^{(j)})}\nonumber\\
&	\quad\times\frac{1}{m_1^{(1)}(m_2^{(1)}-m_1^{(1)})\cdot\ldots\cdot(m^{(1)}_{s_1}-m_{s_1-1}^{(1)})\cdots m_1^{(b)}(m_2^{(b)}-m_1^{(b)})\cdot\ldots\cdot(m_{s_b}^{(b)}-m_{s_b-1}^{(b)})}\\
&	\quad\times\frac{\dbinom{2(l_{j+1}-m_{s_1}^{(1)})}{l_{j+1}-m_{s_1}^{(1)}}}{2^{2(l_{j+1}-m_{s_1}^{(1)})}}\cdot\ldots\cdot
\frac{\dbinom{2(l_{j+b}-m_{s_b}^{(b)})}{l_{j+b}-m_{s_b}^{(b)}}}{2^{2(l_{j+b}-m_{s_b}^{(b)})}}\nonumber.
\end{align}

Now, we can compute the sum of the $k$-th conic intrinsic volumes of $T_l$ over all $l$:
\begin{align*}
&\sum_{\substack{l_1,\dots,l_j\in\N,l_{j+1},\dots,l_{j+b}\in\N_0:\\l_1+\ldots+l_{j+b}=n}}\upsilon_k(T_l)\\
&	=\sum_{\substack{l_1,\dots,l_j\in\N,l_{j+1},\dots,l_{j+b}\in\N_0:\\l_1+\ldots+l_{j+b}=n}}\sum_{F\in\cF_k(T_l)}\alpha(F)\alpha(N_F(T_l))\\
&	=\sum_{\substack{l_1,\dots,l_j\in\N,l_{j+1},\dots,l_{j+b}\in\N_0:\\l_1+\ldots+l_{j+b}=n}}\:\sum_{\substack{r_1,\dots,r_j\in\N,s_1,\dots,s_b\in\N_0:\\r_1+\ldots+r_j+s_1+\ldots+s_b=k}}\:\sum_{(i,m)\text{ satisfying }\eqref{eq:cond_l1}}\alpha(F_{l,i,m})\alpha(N_{F_{l,i,m}}(T_l))\\
&	=\sum_{\substack{r_1,\dots,r_j\in\N,s_1,\dots,s_b\in\N_0:\\r_1+\ldots+r_j+s_1+\ldots+s_b=k}}\:\sum_{\substack{l_1,\dots,l_j\in\N,l_{j+1},\dots,l_{j+b}\in\N_0:\\l_1+\ldots+l_{j+b}=n}}\:\sum_{(i,m)\text{ satisfying }\eqref{eq:cond_l1}}\alpha(F_{l,i,m})\alpha(N_{F_{l,i,m}}(T_l)).
\end{align*}
Let us introduce the notation for the lengths of the bridges and walks. For the lenghts of the bridges corresponding to $A$-faces, we write
\begin{multline*}
h_1=i_1^{(1)},h_2=i_2^{(1)}-i_1^{(1)},\dots,h_{r_1}=l_1-i_{r_1-1}^{(1)},\dots,\\h_{r_1+\ldots+r_{j-1}+1}=i_1^{(j)},
h_{r_1+\ldots+r_{j-1}+2}=i_2^{(j)}-i_1^{(j)},\dots,h_{r_1+\ldots+r_{j}}=l_j-i_{r_j-1}^{(j)}.
\end{multline*}
Similarly, the lengths of the  bridges corresponding to $B$-faces are denoted by
\begin{multline*}
h_{r_1+\ldots+r_j+1}=m_1^{(1)},h_{r_1+\ldots+r_j+2}=m_2^{(1)}-m_1^{(1)},\dots,h_{r_1+\ldots+r_j+s_1}=m_{s_1}^{(1)}-m_{s_1-1}^{(1)},\\\dots,h_{r_1+\ldots+r_j+s_1+\ldots+s_b}=h_k=m_{s_b}^{(b)}-m_{s_b-1}^{(b)}.
\end{multline*}
All these numbers are in $\N$. Additionally,  the lengths of the walks corresponding to the $B$-faces are  denoted by
\begin{align*}
h_{k+1}=l_{j+1}-m_{s_1}^{(1)},h_{k+2}=l_{j+2}-m_{s_2}^{(2)},\dots,  h_{k+b}= l_{j+b}-m_{s_b}^{(b)}.
\end{align*}
These numbers are in $\N_0$.
Now, we can change the summation and arrive at
\begin{align}\label{eq:sum1}
&\sum_{\substack{r_1,\dots,r_j\in\N,s_1,\dots,s_b\in\N_0\\r_1+\ldots+r_j+s_1+\ldots+s_b=k}}\:\sum_{\substack{h_1,\dots,h_k\in\N,h_{k+1},\dots,h_{k+b}\in\N_0\\h_1+\ldots+h_{k+b}=n}}\frac{\binom{2h_{k+1}}{h_{k+1}}\cdot\ldots\cdot\binom{2h_{k+b}}{h_{k+b}}}{h_1h_2\cdot\ldots\cdot h_k\cdot 2^{2(h_{k+1}+\ldots+h_{k+b})}}\notag\\
&\quad\times \P\Big(\frac{\xi_1}{\sqrt{h_1}}\ge\ldots\ge \frac{\xi_{r_1}}{\sqrt{h_{r_1}}}\Big)\cdot\ldots\cdot\P\Big(\frac{\xi_1}{\sqrt{h_{r_1+\ldots+r_{j-1}+1}}}\ge\ldots\ge \frac{\xi_{r_j}}{\sqrt{h_{r_1+\ldots+r_j}}}\Big)\\
&	\quad\times\P\Big(\frac{\xi_1}{\sqrt{h_{r_1+\ldots+r_j+1}}}\ge\ldots\ge \frac{\xi_{s_1}}{\sqrt{h_{r_1+\ldots+r_j+s_1}}}\ge 0\Big)\notag\\
&	\quad\times\ldots\times\P\Big(\frac{\xi_1}{\sqrt{h_{r_1+\ldots+r_j+s_1+\ldots+s_{b-1}+1}}}\ge\ldots\ge \frac{\xi_{s_b}}{\sqrt{h_{r_1+\ldots+r_j+s_1+\ldots+s_b}}}\ge 0\Big)\notag
\end{align}
using the explicit representation of the internal angles in~\eqref{eq:int_angles} and external angles in~\eqref{eq:ext_angles}.

After introducing additional (signed) permutations of the $\xi_i$'s, we will arrange
the probabilities in the above sum in groups that sum up to $1$. In order to explain this, consider a permutation $\pi\in\text{Sym}(r_1)$. For each tuple $(h_1,\dots,h_{k+b})$ such that $h_1+\ldots+h_{k+b}=n$, the tuple
$(h_{\pi(1)},\dots,h_{{\pi(r_1)}}, h_{r_1+1},\ldots, h_{k+b})$  (in which we permuted the  first $r_1$ coordinates according to $\pi$) satisfies the same condition. Thus the above sum does not change if we replace $(h_1,\dots,h_{r_1})$ by $(h_{\pi(1)},\dots,h_{{\pi(r_1)}})$ inside the summation. Additionally, we observe that
\begin{align*}
\sum_{\pi\in\text{Sym}(r_1)}\P\Big(\frac{\xi_{\pi(1)}}{\sqrt{h_{\pi(1)}}}\ge\ldots\ge \frac{\xi_{\pi(r_1)}}{\sqrt{h_{\pi(r_1)}}}\Big)=1
\end{align*}
and therefore also
\begin{align*}
\sum_{\pi\in\text{Sym}(r_1)}\P\Big(\frac{\xi_{1}}{\sqrt{h_{\pi(1)}}}\ge\ldots\ge \frac{\xi_{r_1}}{\sqrt{h_{\pi(r_1)}}}\Big)=1
\end{align*}
holds true. The same argument can be applied for the first $j$ probabilities in~\eqref{eq:sum1}. For the other $b$ probabilities, we argue in a similar way using signed permutations.  For example, we have
\begin{align*}
\sum_{(\varepsilon,\pi)\in\{\pm 1\}^{s_1}\times\text{Sym}(s_1)}
\P\bigg(\frac{\varepsilon_1\xi_{\pi(1)}}{\sqrt{h_{r_1+\ldots+r_j+\pi(1)}}}\ge\ldots\ge \frac{\varepsilon_{s_1}\xi_{\pi(s_1)}}{\sqrt{h_{r_1+\ldots+r_j+\pi(s_1)}}}\ge 0\bigg)=1
\end{align*}
and thus, also
\begin{align*}
\sum_{(\varepsilon,\pi)\in\{\pm 1\}^{s_1}\times\text{Sym}(s_1)}
\P\bigg(\frac{\varepsilon_1\xi_{1}}{\sqrt{h_{r_1+\ldots+r_j+\pi(1)}}}\ge\ldots\ge \frac{\varepsilon_{s_1}\xi_{s_1}}{\sqrt{h_{r_1+\ldots+r_j+\pi(s_1)}}}\ge 0\bigg)=1.
\end{align*}
Thus, we can rewrite~\eqref{eq:sum1} to obtain
\begin{align}\label{eq:sum2}
&\sum_{\substack{r_1,\dots,r_j\in\N,s_1,\dots,s_b\in\N_0\\r_1+\ldots+r_j+s_1+\ldots+s_b=k}}\frac{1}{r_1!\cdot\ldots\cdot r_j!s_1!\cdot\ldots\cdot s_b!2^{s_1+\ldots+s_b}}\\
&\times\sum_{\substack{h_1,\dots,h_k\in\N,h_{k+1},\dots,h_{k+b}\in\N_0\\h_1+\ldots+h_{k+b}=n}}\frac{\binom{2h_{k+1}}{h_{k+1}}\cdot\ldots\cdot\binom{2h_{k+b}}{h_{k+b}}}{h_1h_2\cdot\ldots\cdot h_k\cdot 2^{2(h_{k+1}+\ldots+h_{k+b})}}.\notag
\end{align}
The  sums can be evaluated  separately. We start with the first sum.
By summing over the possible values $m$ of $s_1+\ldots+s_b$, we get
\begin{align}\label{eq:sum3}
&\sum_{\substack{r_1,\dots,r_j\in\N,s_1,\dots,s_b\in\N_0\\r_1+\ldots+r_j+s_1+\ldots+s_b=k}}
\frac{1}{r_1!\cdot\ldots\cdot r_j!s_1!\cdot\ldots\cdot s_b!2^{s_1+\ldots+s_b}}\notag\\
&	\quad=\sum_{m=0}^{k-j}\sum_{\substack{r_1,\dots,r_j\in\N:\\r_1+\ldots+r_j+m=k}}\frac{1}{r_1!\cdot\ldots\cdot r_j!}\sum_{\substack{s_1,\dots,s_b\in\N_0:\\s_1+\ldots+s_b=m}}\frac{1}{2^ms_1!\cdot\ldots\cdot s_b!}\notag\\
&	\quad=\sum_{m=0}^{k-j}\sum_{\substack{r_1,\dots,r_j\in\N:\\r_1+\ldots+r_j+m=k}}\frac{1}{r_1!\cdot\ldots\cdot r_j!m!2^m}\sum_{\substack{s_1,\dots,s_b\in\N_0:\\s_1+\ldots+s_b=m}}\binom{m}{s_1,\dots,s_b}\notag\\
&	\quad=\sum_{m=0}^{k-j}\sum_{\substack{r_1,\dots,r_j\in\N:\\r_1+\ldots+r_j=k-m}}\frac{b^m}{r_1!\cdot\ldots\cdot r_j!m!2^m},
\end{align}
where we used the multinomial theorem in the last step. Using the representation in~\eqref{eq:rep_stirling2_sum} of the Stirling number of the second kind, we can rewrite~\eqref{eq:sum3} to obtain
\begin{align*}
\sum_{m=0}^{k-j}\Big(\frac{b}{2}\Big)^m\frac{j!}{m!(k-m)!}\stirlingsec{k-m}{j}
&	=\frac{j!}{k!}\sum_{m=0}^{k-j}\Big(\frac{b}{2}\Big)^m\binom{k}{k-m}\stirlingsec{k-m}{j}\\
&	=\frac{j!}{k!}\sum_{m=j}^k\binom{k}{m}\stirlingsec{m}{j}\Big(\frac{b}{2}\Big)^{k-m}\\
&	=\frac{j!}{k!}\stirlingsec{k+b/2}{j+b/2}_{b/2},
\end{align*}
where the last step follows from the definition of the $r$-Stirling numbers of the second kind in~\eqref{eq:def_analy_r-stirling2}.

The second sum of~\eqref{eq:sum2} can be rewritten as follows:
\begin{align}\label{eq:sum4}
\sum_{m=0}^{n-k}\sum_{\substack{h_1,\dots,h_k\in\N:\\h_1+\ldots+h_k+m=n}}\frac{1}{h_1\cdot\ldots\cdot h_k}\sum_{\substack{h_{k+1},\dots,h_{k+b}\in\N_0:\\h_{k+1}+\ldots+h_{k+b}=m}}\frac{1}{2^{2m}}\binom{2h_{k+1}}{h_{{k+1}}}\cdot\ldots\cdot \binom{2h_{k+b}}{h_{{k+b}}}.
\end{align}
A generalization of Vandermonde's convolution~\cite[(3.27)]{Graham1994} yields
\begin{align}\label{eq:vandermonde}
\sum_{\substack{h_{k+1},\dots,h_{k+b}\in\N_0:\\h_{k+1}+\ldots+h_{k+b}=m}}\binom{-1/2}{h_{k+1}}\cdot\ldots\cdot\binom{-1/2}{h_{k+b}}=\binom{-b/2}{m}.
\end{align}
Using the formula
\begin{align*}
\binom{-1/2}{d}=\Big(-\frac{1}{4}\Big)^d\binom{2d}{d},\quad d\in\N_0;
\end{align*}
see~\cite[(5.37)]{Graham1994}, we have
\begin{align*}
\binom{-1/2}{h_{k+1}}\cdot\ldots\cdot\binom{-1/2}{h_{k+b}}
&	=\frac{(-1)^m}{4^m}\binom{2h_{k+1}}{h_{k+1}}\cdot\ldots\cdot\binom{2h_{k+b}}{h_{k+b}},
\end{align*}
for all $h_{k+1},\dots,h_{k+b}\in\N_0$ such that $h_{k+1}+\ldots+h_{k+b}=m$. Combining this with~\eqref{eq:vandermonde}, we obtain
\begin{align*}
\sum_{\substack{h_{k+1},\dots,h_{k+b}\in\N_0:\\h_{k+1}+\ldots+h_{k+b}=m}}\binom{2h_{k+1}}{h_{k+1}}\cdot\ldots\cdot\binom{2h_{k+b}}{h_{k+b}}
&	=(-4)^m\sum_{\substack{h_{k+1},\dots,h_{k+b}\in\N_0:\\h_{k+1}+\ldots+h_{k+b}=m}}\binom{-1/2}{h_{k+1}}\cdot\ldots\cdot\binom{-1/2}{h_{k+b}}\\
&	=(-4)^m\binom{-b/2}{m}.
\end{align*}
Applying this to~\eqref{eq:sum4} yields
\begin{align*}
&\sum_{m=0}^{n-k}\sum_{\substack{h_1,\dots,h_k\in\N:\\h_1+\ldots+h_k+m=n}}\frac{1}{h_1\cdot\ldots\cdot h_k}\frac{1}{2^{2m}}(-4)^m\binom{-b/2}{m}\\
&	\quad=\sum_{m=0}^{n-k}(-1)^m\binom{-b/2}{m}\sum_{\substack{h_1,\dots,h_k\in\N:\\h_1+\ldots+h_k=n-m}}\frac{1}{h_1\cdot\ldots\cdot h_k}\\
&	\quad=\sum_{m=0}^{n-k}\frac{b}{2}\Big(\frac{b}{2}+1\Big)\cdot\ldots\cdot\Big(\frac{b}{2}+m-1\Big)\frac{k!}{m!(n-m)!}\stirling{n-m}{k}\\
&	\quad=\frac{k!}{n!}\sum_{m=0}^{n-k}\binom{n}{m}\stirling{n-m}{k}\frac{b}{2}\Big(\frac{b}{2}+1\Big)\cdot\ldots\cdot\Big(\frac{b}{2}+m-1\Big)\\
&	\quad=\frac{k!}{n!}\stirling{n+b/2}{k+b/2}_{b/2},
\end{align*}
where we used the representation in~\eqref{eq:rep_stirling1_sum} of the Stirling numbers of the first kind and the definition of the $r$-Stirling numbers of the first kind in~\eqref{eq:def_analy_r-stirling1}. If we insert both sums in~\eqref{eq:sum2}, we finally obtain
\begin{align*}
\sum_{\substack{l_1,\dots,l_j\in\N,l_{j+1},\dots,l_{j+b}\in\N_0:\\l_1+\ldots+l_{j+b}=n}}\upsilon_k(T_l)
&	=\frac{j!}{k!}\stirlingsec{k+b/2}{j+b/2}_{b/2}\frac{k!}{n!}\stirling{n+b/2}{k+b/2}_{b/2}
=\frac{j!}{n!}\stirling{n+b/2}{k+b/2}_{b/2}\stirlingsec{k+b/2}{j+b/2}_{b/2},
\end{align*}
which completes the proof.
\end{proof}

\subsection{Proof of Theorem~\ref{theorem:prod_typeB_simplex}}\label{section:proof_prod_theorems}

Let $b\in\N$. Recall that $K^B=K_{n_1}^B\times\ldots\times K_{n_b}^B$ and $K^A=K_{n_1}^A\times\ldots\times K_{n_b}^A$ for $n_1,\ldots,n_b\in\N_0$ such that $n:=n_1+\ldots+n_b$,
where
\begin{align*}
K_d^B&=\{x\in\R^d:1\ge x_1\ge x_2\ge \ldots\ge x_d\ge 0\},\\
K_{d}^A&=\{x\in\R^{d+1}:x_1\ge\ldots\ge x_{d+1}, x_1-x_{d+1}\le 1, x_1+\ldots+x_{d+1}=0\},
\end{align*}
for $d\in\N$, denote the Schl\"afli orthoschemes of types $B$ and $A$ in $\R^d$, respectively $\R^{d+1}$. For convenience, we set $K_0^B=K_0^A=\{0\}$. 
We want to show that
\begin{align*}
\sum_{F\in\cF_j(K^B)}\upsilon_k(T_F(K^B))=\sum_{F\in\cF_j(K^A)}\upsilon_k(T_F(K^A))=R_1\big(k,j,b,(n_1,\ldots,n_b)\big)
\end{align*}
holds for all $j\in\{0,\ldots,n\}$ and $k\in\{0,\ldots,n\}$, where for $d\in\{0,\frac{1}{2},1\}$ we define
\begin{align*}
&R_d\big(k,j,b,(n_1,\ldots,n_b)\big)\\
&	\quad:=\big[t^k\big]\big[x_1^{n_1}\cdot\ldots\cdot x_b^{n_b}\big]\big[u^j\big]\frac{(1-x_1)^{-d(t+1)}\cdot\ldots\cdot (1-x_b)^{-d(t+1)}}{\big(1-u((1-x_1)^{-t}-1)\big)\cdot\ldots\cdot \big(1-u((1-x_b)^{-t}-1)\big)}.
\end{align*}

\begin{proof}[Proof of Theorem~\ref{theorem:prod_typeB_simplex}]
We divide the proof into $3$ steps. In the first step, we describe the tangent cones $T_F(K^B)$ in terms of products of Weyl chambers. In Step $2$, we derive a formula for the generalized angle sums of the $T_F(K^B)$'s and show that the derived formula simplifies to the desired constant $R_1(k,j,b,(n_1,\ldots,n_b))$. In the third step, following the arguments of the $B$-case, we write the conic intrinsic volumes of the tangent cones $T_F(K^A)$ in terms of products of Weyl chambers, counted with a certain multiplicity, and show that the formula for the generalized angle sums in the $B$-case also holds in the $A$-case.

\vspace*{2mm}
\noindent
\textsc{Step 1.}
Let $j\in\{0,\ldots,n\}$ and $k\in\{0,\ldots,n\}$ be given. It is easy to check that the constant $R_d\big(k,j,b,(n_1,\ldots,n_b)\big)$ vanishes for $k<j$, so that we need to consider the case $k\geq j$ only. For each $j$-face $F$ of $K^B$, there are numbers $j_1,\ldots,j_b\in\N_0$ satisfying $j_1+\ldots+j_b=j$, such that
\begin{align*}
F=F_1\times\ldots\times F_b
\end{align*}
for some $F_1\in\cF_{j_1}(K_{n_1}^B),\ldots, F_b\in\cF_{j_b}(K_{n_b}^B)$. Thus, the tangent cone of $K^B$ at $F$ is given by the following product formula:
\begin{align*}
T_F(K^B)=T_{F_1}(K_{n_1}^B)\times\ldots\times T_{F_b}(K_{n_b}^B).
\end{align*}
In order to see this, observe that $\relint F=\relint F_1 \times\ldots\times \relint F_b$. Take $x=(x^{(1)},\ldots,x^{(b)})\in\relint F\subset \R^n$, where $x^{(i)}\in\relint F_i\subset\R^{n_i}$ for $i=1,\ldots,b$. Then, we have
\begin{align*}
&T_F(K^B)\\
&	\quad=\{v\in\R^n:\exists \eps >0 \text{ with } x+\eps v\in K^B\}\\
&	\quad=\{v\in\R^{n_1}\times\ldots\times\R^{n_b}:\exists \eps >0 \text{ with } x+\eps v\in K_{n_1}^B\times\ldots\times K_{n_b}^B\}\\
&	\quad=\{v_1\in\R^{n_1}:\exists \eps >0 \text{ with }x^{(1)}+\eps v_1\in K_{n_1}^B\}\times\ldots\times\{v_b\in\R^{n_b}:\exists \eps >0 \text{ with }x^{(b)}+\eps v_b\in K^B_{n_b}\}\\
&	\quad=T_{F_1}(K_{n_1}^B)\times\ldots\times T_{F_b}(K_{n_b}^B).
\end{align*}
Applying Lemma~\ref{lemma:tangent_cones_typeB} to the individual terms in the product, we observe that the collection
\begin{align*}
T_{F_1}(K_{n_1}^B)\times\ldots\times T_{F_b}(K_{n_b}^B),
\end{align*}
where $F_1\in\cF_{j_1}(K_{n_1}^B),\ldots, F_b\in\cF_{j_b}(K_{n_b}^B)$, coincides (up to isometries) with the collection of cones
\begin{align*}
G_i:=A^{(i_1^{(1)})}\times \ldots\times A^{(i_{j_1}^{(1)})}\times\ldots\times A^{(i_1^{(b)})}\times\ldots\times A^{(i_{j_b}^{(b)})}\times B^{(i_0^{(1)})}\times B^{(i_{j_1+1}^{(1)})}\times\ldots\times B^{(i_0^{(b)})}\times B^{(i_{j_b+1}^{(b)})}
\end{align*}
where $i_1^{(1)},\ldots,i_{j_1}^{(1)},\ldots,i_1^{(b)},\ldots,i_{j_b}^{(b)}\in\N$ and $i_0^{(1)},i_{j_1+1}^{(1)},\ldots,i_0^{(b)},i_{j_b+1}^{(b)}\in\N_0$ such that
\begin{align*}
i_0^{(1)}+\ldots+i_{j_1+1}^{(1)}=n_1,\ldots,i_0^{(b)}+\ldots+i_{j_b+1}^{(b)}=n_b.
\end{align*}
This yields
\begin{align}\label{eq:important_typeB}
\sum_{F\in\cF_j(K^B)}\upsilon_{k}(T_F(K^B))
&	=\sum_{\substack{j_1,\ldots,j_b\in\N_0:\\j_1+\ldots+j_b=j}}\:\sum_{F_1\in\cF_{j_1}(K_{n_1}^B),\ldots, F_b\in\cF_{j_b}(K_{n_b}^B)}\upsilon_k \big(T_{F_1}(K_{n_1}^B)\times\ldots\times T_{F_b}(K_{n_b}^B)\big)\notag\\
&	=\sum_{\substack{j_1,\ldots,j_b\in\N_0:\\j_1+\ldots+j_b=j}}\:\sum_{\substack{i_0^{(1)}+\ldots+i_{j_1+1}^{(1)}=n_1\\\ldots\\i_0^{(b)}+\ldots+i_{j_b+1}^{(b)}=n_b}}\upsilon_k(G_i).
\end{align}

\vspace*{2mm}
\noindent
\textsc{Step 2.}
Similarly to the proof of Proposition~\ref{prop:sum_intr_vol}, we observe that $\nu_k(G_i)$  is the coefficient of $t^k$ in the polynomial
\begin{multline}\label{eq:aaaaa}
\frac{t^{\overline{i_1^{(1)}}}}{i_1^{(1)}!}\cdot\ldots\cdot \frac{t^{\overline{i_{j_1}^{(1)}}}}{i_{j_1}^{(1)}!}
\cdot\ldots\cdot \frac{t^{\overline{i_1^{(b)}}}}{i_1^{(b)}!}\cdot\ldots\cdot \frac{t^{\overline{i_{j_b}^{(b)}}}}{i_{j_b}^{(b)}!}\\
\times\frac{(t+1)(t+3)\cdot\ldots\cdot (t+2i_0^{(1)}-1)}{2^{i_0^{(1)}}i_0^{(1)}!}\cdot\frac{(t+1)(t+3)\cdot\ldots\cdot (t+2i_{j_1+1}^{(1)}-1)}{2^{i_{j_1+1}^{(1)}}i_{j_1+1}^{(1)}!}\\
\times\ldots\times\frac{(t+1)(t+3)\cdot\ldots\cdot (t+2i_0^{(b)}-1)}{2^{i_0^{(b)}}i_0^{(b)}!}\cdot \frac{(t+1)(t+3)\cdot\ldots\cdot (t+2i_{j_b+1}^{(b)}-1)}{2^{i_{j_b+1}^{(b)}}i_{j_b+1}^{(b)}!}.
\end{multline}
Following the same arguments as in the proof of Proposition~\ref{prop:sum_intr_vol}, we obtain that
\begin{align*}
\sum_{F\in\cF_j(K^B)}\upsilon_{k}(T_F(K^B)) =\big[t^k]\sum_{\substack{j_1,\ldots,j_b\in\N_0:\\j_1+\ldots+j_b=j}}\:
\sum_{\substack{i_0^{(1)}+\ldots+i_{j_1+1}^{(1)}=n_1\\\ldots\\i_0^{(b)}+\ldots+i_{j_b+1}^{(b)}=n_b}}
\text{term in~\eqref{eq:aaaaa}}.
\end{align*}
We introduce new variables $x_1,\ldots,x_b$ and expand the product to write the right-hand side as follows:
\begin{multline*}
\big[t^k]\sum_{\substack{j_1,\ldots,j_b\in\N_0:\\j_1+\ldots+j_b=j}}\Bigg(\big[x_1^{n_1}\big]\bigg(\sum_{l=1}^\infty\frac{t^{\overline{l}}}{l!}x_1^l\bigg)^{j_1}\bigg(\sum_{l=0}^\infty \frac{(t+1)(t+3)\cdot\ldots\cdot (t+2l-1)}{2^ll!}x_1^l\bigg)^2\times\ldots\\
\times\big[x_b^{n_b}\big]\bigg(\sum_{l=1}^\infty\frac{t^{\overline{l}}}{l!}x_b^l\bigg)^{j_b}\bigg(\sum_{l=0}^\infty \frac{(t+1)(t+3)\cdot\ldots\cdot (t+2l-1)}{2^ll!}x_b^l\bigg)^2\Bigg).
\end{multline*}
Using the formulas~\eqref{eq:sum/l!} and~\eqref{eq:sum_B} we arrive at  
\begin{align}\label{eq:further_compute}
\sum_{F\in\cF_j(K^B)}\upsilon_{k}(T_F(K^B))
&	=\big[t^k\big]\sum_{\substack{j_1,\ldots,j_b\in\N_0:\\j_1+\ldots+j_b=j}}\bigg(\big[x_1^{n_1}\big]\left((1-x_1)^{-(t+1)}\left((1-x_1)^{-t}-1\right)^{j_1}\right)\notag\\
&	\hspace*{2.5cm}\times\ldots\times\big[x_b^{n_b}\big]\left((1-x_b)^{-(t+1)}\left((1-x_b)^{-t}-1\right)^{j_b}\right)\bigg)\notag\\
&	=
\big[t^k\big]\big[x_1^{n_1}\cdot\ldots\cdot x_b^{n_b}\big]
\sum_{\substack{j_1,\ldots,j_b\in\N_0:\\j_1+\ldots+j_b=j}}
\bigg((1-x_1)^{-(t+1)}\cdot\ldots\cdot(1-x_b)^{-(t+1)}  \notag\\
&	\quad\times\left((1-x_1)^{-t}-1\right)^{j_1}\cdot\ldots\cdot \left((1-x_b)^{-t}-1\right)^{j_b}\bigg).
\end{align}
By introducing a new variable $u$ and expanding the product again, we arrive at
\begin{align*}
\sum_{F\in\cF_j(K^B)}\upsilon_{k}(T_F(K^B))
&	=\big[t^k\big]\big[x_1^{n_1}\cdot\ldots\cdot x_b^{n_b}\big]\big[u^j\big]\left((1-x_1)^{-(t+1)}\cdot\ldots\cdot(1-x_b)^{-(t+1)} \right) \\
&	\quad\times\bigg(\sum_{l=0}^\infty \big((1-x_1)^{-t}-1\big)^lu^l\bigg)\cdot\ldots\cdot\bigg(\sum_{l=0}^\infty \big((1-x_b)^{-t}-1\big)^lu^l\bigg)\\
&	=\big[t^k\big]\big[x_1^{n_1}\cdot\ldots\cdot x_b^{n_b}\big]\big[u^j\big]
\left((1-x_1)^{-(t+1)}\cdot\ldots\cdot(1-x_b)^{-(t+1)}  \right)\\
&	\quad\times\frac{1}{\big(1-u((1-x_1)^{-t}-1)\big)\cdot\ldots\cdot\big(1-u((1-x_b)^{-t}-1)\big)}\\
&	=R_1\big(k,j,b,(n_1,\ldots,n_b)\big),
\end{align*}
which completes the proof of the $B$-case.

\vspace*{2mm}
\noindent
\textsc{Step 3.} In the $A$-case, instead of considering $K^A$, we look at $\widetilde K^A=\widetilde K^A_{n_1}\times\ldots\times\widetilde K^A_{n_b}$ for $n_1,\ldots,n_b\in\N_0$ such that $n:=n_1+\ldots+n_b$, where
\begin{align*}
\widetilde{K}_{d}^A=\{x\in\R^{d+1}:x_1\ge\ldots\ge x_{d+1}, x_1-x_{d+1}\le 1\},\quad d\in\N,
\end{align*}
denotes the unbounded polyhedral set related to $K_{d}^A$. Note that $\widetilde K_0^A=\R$.


Following the arguments of Step $1$, we can write the tangent cones of $T_F(K^A)$ (and also $T_F(\widetilde K^A)$) as products of tangent cones at their respective faces.
Using $\upsilon_k(K_n^A)=\upsilon_{k+1}(K_n^A\oplus L_{n+1})$ and~\eqref{eq:tangent_cones_transform}, we obtain
\begin{align}\label{eq:sum_rewrite1}
&\sum_{F\in\F_j(K^A)}\upsilon_k(T_F(K^A))\notag\\
&	\quad=\sum_{\substack{j_1,\ldots,j_b\in\N_0:\\j_1+\ldots+j_b=j}}\:\sum_{F_1\in\cF_{j_1}(K_{n_1}^A),\ldots,F_b\in\cF_{j_b}(K_{n_b}^A)}\upsilon_k\big(T_{F_1}(K_{n_1}^A)\times\ldots\times T_{F_b}(K_{n_b}^A)\big)\notag\\
&	\quad=\sum_{\substack{j_1,\ldots,j_b\in\N_0:\\j_1+\ldots+j_b=j}}\:\sum_{F_1\in\cF_{j_1+1}(\widetilde{K}_{n_1}^A),\ldots,F_b\in\cF_{j_b+1}(\widetilde{K}_{n_b}^A)}\upsilon_{k+b}\big(T_{F_1}(\widetilde{K}_{n_1}^A)\times\ldots\times T_{F_b}(\widetilde{K}_{n_b}^A)\big).
\end{align}
Applying Lemma~\ref{lemma:collection_tangentcones_A} to each individual tangent cone in the product, we see that the collection of tangent cones $T_{F_1}(\widetilde{K}_{n_1}^A)\times\ldots\times T_{F_b}(\widetilde{K}_{n_b}^A)$, where $F_1\in \cF_{j_1+1}(\widetilde{K}_{n_1}^A),\ldots,F_b\in \cF_{j_b+1}(\widetilde{K}_{n_b}^A)$, coincides (up to isometry) with the collection
\begin{align*}
A^{(i_1^{(1)})}\times\ldots\times A^{(i_{j_1+1}^{(1)})}\times\ldots\times A^{(i_1^{(b)})}\times\ldots\times A^{(i_{j_b+1}^{(b)})},
\end{align*}
where $i_1^{(1)},\ldots,i_{j_1+1}^{(1)},\ldots,i_1^{(b)},\ldots,i_{j_b+1}^{(b)}\in \N$ such that
\begin{align*}
i_1^{(1)}+\ldots+i_{j_1+1}^{(1)}=n_1+1,\ldots,i_1^{(b)}+\ldots+i_{j_b+1}^{(b)}=n_b+1
\end{align*}
and each cone of the above collection is taken with multiplicity $i_1^{(1)}\cdot\ldots\cdot i_1^{(b)}$. Thus, the formula in~\eqref{eq:sum_rewrite1} can be rewritten as
\begin{align}\label{eq:sum123}
\sum_{\substack{j_1,\ldots,j_b\in\N_0:\\j_1+\ldots+j_b=j}}\:\sum_{\substack{i_1^{(1)}+\ldots+i_{j_1+1}^{(1)}=n_1+1\\\ldots\\i_1^{(b)}+\ldots+i_{j_b+1}^{(b)}=n_b+1}}i_1^{(1)}\cdot\ldots\cdot i_1^{(b)}\upsilon_{k+b}\big(A^{(i_1^{(1)})}\times\ldots\times A^{(i_{j_b+1}^{(b)})}\big)
\end{align}
for $k\in\{b,\ldots,n+b\}$ and $j\in\{b,\ldots,n+b\}$. The conic intrinsic volume on the right hand side of~\eqref{eq:sum123} is given as the coefficient of $t^{k+b}$ in the following polynomial:
\begin{align*}
\frac{t^{\overline{i_1^{(1)}}}}{i_1^{(1)}!}\cdot\ldots\cdot \frac{t^{\overline{i_{j_1+1}^{(1)}}}}{i_{j_1+1}^{(1)}!}\times\ldots\times \frac{t^{\overline{i_1^{(b)}}}}{i_1^{(b)}!}\cdot\ldots\cdot \frac{t^{\overline{i_{j_b+1}^{(b)}}}}{i_{j_b+1}^{(b)}!}.
\end{align*}
Thus, the term in~\eqref{eq:sum123} simplifies to
\begin{align*}
&\big[t^{k+b}\big]\sum_{\substack{j_1,\ldots,j_b\in\N_0:\\j_1+\ldots+j_b=j}}\:\sum_{\substack{i_1^{(1)}+\ldots+i_{j_1+1}^{(1)}=n_1+1\\\ldots\\i_1^{(b)}+\ldots+i_{j_b+1}^{(b)}=n_b+1}}\Bigg(\frac{t^{\overline{i_1^{(1)}}}}{(i_1^{(1)}-1)!}\frac{t^{\overline{i_2^{(1)}}}}{i_2^{(1)}!}\cdot\ldots\cdot \frac{t^{\overline{i_{j_1+1}^{(1)}}}}{i_{j_1+1}^{(1)}!}\times\ldots\times \frac{t^{\overline{i_1^{(b)}}}}{(i_1^{(b)}-1)!}\frac{t^{\overline{i_2^{(b)}}}}{i_2^{(b)}!}\cdot\ldots\cdot \frac{t^{\overline{i_{j_b+1}^{(b)}}}}{i_{j_b+1}^{(b)}!}\Bigg)\\
&	\quad=\big[t^{k+b}\big]\sum_{\substack{j_1,\ldots,j_b\in\N_0:\\j_1+\ldots+j_b=j}}\bigg(\big[x_1^{n_1+1}\big]\left(tx_1(1-x_1)^{-(t+1)}\left((1-x_1)^{-t}-1\right)^{j_1}\right)\\
&	\hspace*{6.5cm}\times\ldots\times \big[x_b^{n_b+1}\big]\left(tx_b(1-x_b)^{-(t+1)}\left((1-x_b)^{-t}-1\right)^{j_b}\right)\bigg)\\
&	\quad=\big[t^k\big]\big[x_1^{n_1}\cdot\ldots\cdot x_b^{n_b}\big](1-x_1)^{-(t+1)}\cdot\ldots\cdot (1-x_b)^{-(t+1)}\\
&	\hspace*{6.1cm}\times\sum_{\substack{j_1,\ldots,j_b\in\N_0:\\j_1+\ldots+j_b=j}}\left(\left((1-x_1)^{-t}-1\right)^{j_1}\cdot\ldots\left((1-x_b)^{-t}-1\right)^{j_b}\right),
\end{align*}
which coincides with~\eqref{eq:further_compute} and therefore completes the proof. Note that we used~\eqref{eq:sum/l!} and
\begin{align*}
\sum_{l=1}^{\infty}\frac{t^{\overline{l}}}{(l-1)!}x^l
	=tx\sum_{l=0}^\infty\frac{(t+1)(t+2)\cdot\ldots\cdot(t+l)}{l!}x^l
	=tx\sum_{l=0}^\infty\binom{-(t+1)}{l}(-x)^l
	=tx(1-x)^{-(t+1)},
\end{align*}
which follows from the binomial series.
\end{proof}

\subsection{Proof of Proposition~\ref{prop:sum_prod_simplices_A&B}}\label{section:proof_prop_sum_prod_a,b}
Our goal is to show that
\begin{align*}
\sum_{\substack{n_1,\ldots,n_b\in\N_0:\\n_1+\ldots+n_b=n}}\sum_{F\in\cF_j(K^B)}\upsilon_k(T_F(K^B))
&	=\sum_{\substack{n_1,\ldots,n_b\in\N_0:\\n_1+\ldots+n_b=n}}\sum_{F\in\cF_j(K^A)}\upsilon_k(T_F(K^A))\\
&	=\frac{j!}{n!}\binom{j+b-1}{b-1}\stirling{n+b}{k+b}_b\stirlingsec{k+b}{j+b}_b
\end{align*}
holds for $j\in\{0,\ldots,n\}$ and $k\in\{0,\ldots,n\}$.

\begin{proof}[Proof of Proposition~\ref{prop:sum_prod_simplices_A&B}]
At first we show the formula for $K^B$. In~\eqref{eq:important_typeB} we saw that
\begin{align*}
\sum_{F\in\cF_j(K^B)}\upsilon_k\big(T_F(K^B)\big)=\sum_{\substack{j_1,\ldots,j_b\in\N_0:\\j_1+\ldots+j_b=j}}\:\sum_{\substack{i_0^{(1)}+\ldots+i_{j_1+1}^{(1)}=n_1\\\ldots\\i_0^{(b)}+\ldots+i_{j_b+1}^{(b)}=n_b}}\upsilon_k(G_i)
\end{align*}
holds true for all $j,k\in\{0,\ldots,n\}$, where
\begin{align*}
G_i=A^{(i_1^{(1)})}\times \ldots\times A^{(i_{j_1}^{(1)})}\times\ldots\times A^{(i_1^{(b)})}\times\ldots\times A^{(i_{j_b}^{(b)})}\times B^{(i_0^{(1)})}\times B^{(i_{j_1+1}^{(1)})}\times\ldots\times B^{(i_0^{(b)})}\times B^{(i_{j_b+1}^{(b)})}.
\end{align*}
Note that in the second sum on the right hand side, the indices satisfy $i_0^{(l)},i_{j_l+1}^{(l)}\in\N_0$ and $i_1^{(l)},\ldots,i_{j_l}^{(l)}\in\N$ for all $l\in\{1,\ldots,b\}$. Thus, we obtain
\begin{align*}
&\sum_{\substack{n_1,\ldots,n_b\in\N_0:\\n_1+\ldots+n_b=n}}\:\sum_{F\in\cF_j(K^B)}\upsilon_k\big(T_F(K^B)\big)\\
&	\quad=\sum_{\substack{j_1,\ldots,j_b\in\N_0:\\j_1+\ldots+j_b=j}}\:\sum_{\substack{n_1,\ldots,n_b\in\N_0:\\n_1+\ldots+n_b=n}}\:\sum_{\substack{i_0^{(1)}+\ldots+i_{j_1+1}^{(1)}=n_1\\\ldots\\i_0^{(b)}+\ldots+i_{j_b+1}^{(b)}=n_b}}\upsilon_k(G_i)\\
&	\quad=\sum_{\substack{j_1,\ldots,j_b\in\N_0:\\j_1+\ldots+j_b=j}}\:\sum_{\substack{l_1,\ldots,l_j\in\N,l_{j+1},\ldots,l_{j+2b}\in\N_0:\\l_1+\ldots+l_{j+2b}=n}}\upsilon_k\big(A^{(l_1)}\times\ldots\times A^{(l_j)}\times B^{(l_{j+1})}\times\ldots\times B^{(l_{j+2b})}\big).
\end{align*}
The last equation follows from simple renumbering. Applying Proposition~\ref{prop:sum_intr_vol} with $b$ replaced by $2b$ yields
\begin{align*}
\sum_{\substack{n_1,\ldots,n_b\in\N_0:\\n_1+\ldots+n_b=n}}\:\sum_{F\in\cF_j(K^B)}\upsilon_k\big(T_F(K^B)\big)
&	=\sum_{\substack{j_1,\ldots,j_b\in\N_0:\\j_1+\ldots+j_b=j}}\frac{j!}{n!}\stirling{n+b}{k+b}_{b}\stirlingsec{k+b}{j+b}_{b}\\
&	=\frac{j!}{n!}\binom{j+b-1}{b-1}\stirling{n+b}{k+b}_{b}\stirlingsec{k+b}{j+b}_{b},
\end{align*}
where we used the well-known fact that the number of compositions of $j$ into $b$ non-negative integers (which may be $0$) is given by $\binom{j+b-1}{b-1}$.
The formula for $K^A$ follows from $$\sum_{F\in\cF_j(K^B)}\upsilon_k(T_F(K^B))=\sum_{F\in\cF_j(K^A)}\upsilon_k(T_F(K^A)),$$
which is due to Theorem~\ref{theorem:prod_typeB_simplex}. This completes the proof.
\end{proof}

\section*{Acknowledgement}
Supported by the German Research Foundation under Germany's Excellence Strategy  EXC 2044 -- 390685587, Mathematics M\"unster: Dynamics - Geometry - Structure  and by the DFG priority program SPP 2265 \textit{Random Geometric Systems}.

\vspace{1cm}

\bibliography{bib}
\bibliographystyle{abbrv}



\end{document}